\documentclass[12pt,centertags,oneside]{amsart}
\usepackage{amstext,amsthm,amscd,typearea,hyperref}
\usepackage{amsmath,amssymb}
\usepackage{a4wide}
\usepackage[mathscr]{eucal}
\usepackage{mathrsfs}
\usepackage{typearea}
\usepackage{charter}
\usepackage{pdfsync}
\usepackage{xcolor}
\usepackage[a4paper,width=16.2cm,top=3cm,bottom=3cm]{geometry}

\numberwithin{equation}{section}
\usepackage{hyperref}

\newtheorem{theorem}{Theorem}[section]
\newtheorem{definition}[theorem]{Definition}
\newtheorem{proposition}[theorem]{Proposition}
\newtheorem{corollary}[theorem]{Corollary}
\newtheorem{lemma}[theorem]{Lemma}
\newtheorem{remark}[theorem]{Remark}

\newtheorem{example}[theorem]{Example}

\newcommand{\supp}{{\rm supp}}

\newcommand{\psh}{\textnormal{PSH}}

\newcommand{\ddc}{dd^c}
\newcommand{\dc}{d^c}

\newcommand{\PSH}{{\rm PSH}}
\newcommand{\QPSH}{{\rm QPSH}}

\newcommand{\capa}{\mathop{\mathrm{Cap}}\nolimits}

\newcommand{\local}{\rm{loc}}

\newcommand{\B}{\mathbb{B}}

\newcommand{\C}{\mathbb{C}}

\newcommand{\N}{\mathbb{N}}

\title[]{\quad Higher complex Sobolev spaces on complex manifolds}
\author{Thai Duong Do\textit{$^{1,2}$} and Duc-Bao Nguyen\textit{$^{3}$}}
\address{\textit{$^{1}$}Department of Mathematics, National University of Singapore - 10, Lower Kent Ridge Road - Singapore 119076\footnote[4]{On leave from Institute of Mathematics, Vietnam Academy of Science and Technology}}
\curraddr{\textit{$^{2}$}VNU University of Engineering and Technology, 144 Xuan Thuy Street, Cau Giay District, Hanoi, 100000, Viet Nam}
\email{dtduong@vnu.edu.vn, duongdothai.vn@gmail.com}
\address{\textit{$^{3}$}Department of Mathematics,  National University of Singapore - 10, Lower Kent Ridge Road - Singapore 119076}
\email{ducbao.nguyen@u.nus.edu}
\begin{document}

%\hyphenpenalty=10000

\date{\today}

\begin{abstract}
We study higher complex Sobolev spaces and their corresponding functional capacities. In particular, we prove the Moser-Trudinger inequality for these spaces and discuss some relationships between these spaces and the complex Monge-Amp\`{e}re equation.
  \end{abstract}

\medskip

\maketitle

\noindent {\bf Classification AMS 2020}: 32Uxx, 32W20, 46E35.

\smallskip

\tableofcontents
\section{Introduction}
 %The complex Sobolev space (introduced by Dinh and Sibony in their seminal paper \cite{DS06}) is an important functional space in the study of complex dynamics in higher dimensions. By studying this space, Dinh and Sibony have shown many significant statistical properties (e.g. the exponential mixing of measures, CLTs for Lipschitzian observables) of mermorphic dynamical systems (see \cite{DS06}). 
 In this paper, we study higher complex Sobolev spaces %introduced by Dinh
 which extend the notion of complex Sobolev space introduced earlier by Dinh-Sibony in \cite{DS_decay}. The complex Sobolev space has been systematically studied by Vigny in \cite{Vigny}.
 The key observation is that this space takes into account the complex structure of the ambient space and is stable under holomorphic transformations; thus, one could view it as a tailored version of the classical Sobolev space to the complex setting. As a consequence, this functional space plays a key role in complex dynamics and it leads to many fruitful applications in this field (e.g., see \cite{Bianchi-Dinh-EquiStateforPkII,dinh2021products,DS_decay,deThelin-Vigny,Vigny_expo-decay-birational,Vu_nonkahler_topo_degree}).
 Not limited to the theory of complex dynamics, the complex Sobolev spaces also find applications in other fields of mathematics. For example, in the study of the complex Monge-Amp\`ere equation, this space has been used as a test space to obtain the necessary and sufficient conditions for whether a given measure is a Monge-Amp\`{e}re measure with H\"{o}lder potentials (see \cite{dinh2022complex,DinhMarinescuVu}). Moreover, it has recently found applications in studying geometric estimates for (singular) K\"{a}hler metrics (see \cite{Vu-log-diameter,Vu-diameter,DNV-local-non-collapsing-bigclass}). 

 The classical Sobolev spaces $W^{n,p}$ are important tools in studying differential equations and are usually used as weak solution spaces for many fundamental equations. An important fact is that when $n$ and $p$ go to infinity, the solution gains more regularity and becomes the %ancient
 classical solution. 
 %So, it is natural to ask.
% \begin{question}\label{mainquestion}
%     Can we generalize the complex Sobolev space to gain more regularity and use it as solution space for complex differential equations?
% \end{question}

We study higher complex Sobolev spaces in both local (i.e., on bounded domains of $\mathbb C^k$) and global (i.e., on compact K\"ahler manifolds) settings. Let $\Omega$ be a bounded domain in $\mathbb{C}^k$, we denote by $\PSH(\Omega)$ the set of plurisubharmonic (psh) functions in $\Omega$. Let $(X,\omega)$ be a compact K\"ahler manifold of dimension $k$, we denote by $\QPSH(X)$ the set of quasi-plurisubharmonic (qpsh) functions on $X$. These functions are locally given as the sum of a smooth and a plurisubharmonic function. We also denote by $\PSH(X,\omega)$ the set of $\omega$-plurisubharmonic ($\omega$-psh) functions on $X$, i.e., the set of qpsh functions $\varphi$ such that $\omega + dd^c \varphi \geq 0$.

\begin{definition}[Higher complex Sobolev spaces]\label{Def higher cplx Sobolev space}
    For $q\geq 1$, we define inductively\break $q$-complex Sobolev spaces $W^{*}_q$ as follows
    \begin{itemize}
        \item[(1)] (local setting) $W^{*}_q(\Omega)$ is the set of all functions $\varphi \in W^{1,2}(\Omega)$ such that $d\varphi\wedge d^c\varphi\leq dd^c\psi$ for some $\psi\in W^{*}_{q-1}(\Omega) \cap \PSH(\Omega)$ ($\psi \in \PSH(\Omega)$ when $q=1$) satisfying $$\int_{\Omega}dd^c \psi \wedge (dd^c |z|^2)^{k-1} < \infty.$$
        \item[(2)] (global setting) $W^{*}_{q}(X)$ is the set of all functions $\varphi \in W^{1,2}(X)$ such that $d\varphi\wedge d^c\varphi\leq c_\varphi\omega+dd^c\psi$ for some constant $c_\varphi\geq0$ and $\psi\in W^{*}_{q-1}(X) \cap \QPSH(X)$ ($\psi \in \QPSH(X)$ when $q=1$).
    \end{itemize}
\end{definition}

In Section~\ref{sec 2}, we will introduce some notions and technical tools to study these functional spaces. In particular,
 using techniques in \cite{DS_decay,Vigny}, we will build a quasinorm $\|\cdot\|_{*,q}$ on $W^*_q$ which makes $W^*_q$ a quasi-Banach space sharing many properties with $W^*_1$. After that, we will prove some basic properties of these spaces and consider some specific examples. Moreover, we also introduce a family of functional capacities and show that all these capacities characterize pluripolar sets, similarly to how the original $W^*_1$-functional capacity does (\cite{Vigny}).

In Section \ref{sec:3}, we will prove the Moser-Trudinger inequalities for these higher complex Sobolev spaces.
 An important point in our results is that the exponent in the Moser-Trudinger inequalities goes to infinity when $q$ goes to infinity. It follows that when $q$ grows, our spaces will gain more regularity and get closer to the bounded functions space.
\vspace{0.2cm}

 \begin{theorem}\label{theoremA}Let $q\geq 1$, $\alpha \in[1,2^q)$ and $K$ be a compact subset of the unit ball $\B$ of $\C^k$. Let $v_1,\ldots,v_k$ be psh functions on $\B$ which are H\"{o}lder continuous of H\"{o}lder exponent $\beta$ for some $\beta \in (0,1)$ with $\|v_j\|_{\mathscr{C}^\beta} \leq 1$ for $1\leq j\leq k$. Let $\varphi \in W^*_q(\B)$ such that $\|\varphi\|_{*,q}\leq 1$. Then there exist strictly positive constants $c_1$ and $c_2$ depending on $K$, $\alpha$, and $\beta$ but independent of $\varphi,v_1,\ldots,v_k$ such that 
    $$\int_K e^{c_1|\varphi|^\alpha}dd^c v_1 \wedge \cdots \wedge dd^c v_k\leq c_2.$$
In particular, 
there exist strictly positive constants $c_1$ and $c_2$ depending on $K$ and $\alpha$ such that for every $\varphi \in W^*_q(\B)$ with $\|\varphi\|_{*,q}\leq 1$, there holds
    $$\int_K e^{c_1|\varphi|^\alpha}(dd^c |z|^2)^k\leq c_2.$$
    \end{theorem}
\vspace{0.2cm}

 Note that two functions in $W^*_q(\B)$ are equal if they are equal almost everywhere. By \cite[Theorem 1.1]{Vigny-Vu-Lebesgue} (see also \cite{DinhMarinescuVu}), for every function $\varphi$ in $W^*_1(\B)$ (and thus $W^*_q(\B)$), all points are Lebesgue except for points in some pluripolar set. Therefore, by considering the canonical values of $\varphi$ at its Lebesgue points, the first integral in the theorem makes sense, as $dd^c v_1 \wedge \cdots \wedge dd^c v_k$ has no mass on pluripolar sets. Throughout this paper, we always use the canonical values of $\varphi$ as above.

 The Moser-Trudinger inequality for $W^*_1$ was proven in \cite{DinhMarinescuVu} by using the slicing method. Recently, in \cite{Vigny-Vu-Lebesgue}, alongside the main goal of proving that the complement of the Lebesgue point set of functions in $W^*_1$ is pluripolar, Vigny-Vu obtained a version of the Moser-Trudinger inequality for $W^*_1$ that corresponds to the last assertion of Theorem \ref{theoremA} for $q=1$.
%We will not use the slicing method; instead, we will use a recent approach of Vigny-Vu in \cite{VV23} for Theorem A in the case of $q=1$. Their first aim was to show that the Lebesgue points set of functions in $W^*_1$ is pluripolar. 
For $\varphi \in W^*_1(\B)$, their strategy was to bound $|\varphi|^\alpha$ by some psh function which allows them to use Skoda's integrability theorem. The construction of the psh bound is motivated by the proof of Josefson's theorem (see \cite{josefson}). Following their strategy, we construct the psh bound for $|\varphi|^\alpha$ where $\varphi \in W^*_q(\B)$ (see Theorem~\ref{main result} below) and deduce Theorem \ref{theoremA} by using a singular version of Skoda's integrability theorem that has been obtained in \cite{DVS_exponential} (see also \cite{Lucas_skodatype}).

%We also obtain the following global version of Theorem \ref{theoremA}.
%\vspace{0.2cm}

%\begin{theorem}[Moser-Trudinger inequality in a global setting]\label{theoremA'}
%     Let $(X,\omega)$ be a compact K\"{a}hler manifold and $\alpha \in[1,2^q)$. Let $v_1,\ldots,v_k$ be $\omega$-psh functions which are H\"{o}lder continuous of H\"{o}lder exponent $\beta$ for some $\beta \in (0,1)$. Let $\varphi \in W^*_q(X)$. Assume that $\|v_j\|_{C^\beta} \leq 1$ for $1\leq j\leq k$ and $\|\varphi\|_{*,q}\leq 1$. Then there exist strictly positive constants $c_1$ and $c_2$ depending on $X, \omega, \alpha$ and $\beta$ but independent of $\varphi,v_1,\ldots,v_k$ such that 
%    $$\int_X e^{c_1|\varphi|^\alpha}(\omega+dd^c v_1)\wedge \cdots \wedge (\omega+dd^c v_k)\leq c_2.$$
%    In particular, there exist strictly positive constants $c_1$ and $c_2$ depending on $X, \alpha,$ and $ \omega$ such that for every $\varphi \in W^*_q(X)$ with $\|\varphi\|_{*,q}\leq 1$, there holds
%    $$\int_X e^{c_1|\varphi|^\alpha}\omega^k\leq c_2.$$
%\end{theorem}
%\vspace{0.2cm}

In Section \ref{sec:4}, we will discuss the connection of the higher complex Sobolev spaces to the theory of complex Monge-Amp\`{e}re equation.
In the global setting, the class $\mathcal{E}(X,\omega)$, introduced by Guedj-Zeriahi in \cite{GZ-weighted}, is the largest class of $\omega$-psh functions on which the complex Monge-Amp\`{e}re operator is well defined and the comparison principle is valid. We will show that $\omega$-psh functions with bounded $\|\cdot\|_{*,q}$-norm belong to this space for every $q \geq 1$. Furthermore, among the subsets of $\mathcal{E}(X,\omega)$, the finite energy classes $\mathcal{E}^p(X,\omega)$ have important applications in the variational approach to complex Monge-Amp\`{e}re equation (see \cite{BBGZ-variational}). In \cite{DGLfinite-entropy}, the authors proved a Moser-Trudinger inequality for functions in $\mathcal{E}^p(X,\omega)$. The crucial point here is that the exponent in their Moser-Trudinger inequality is $1+{p}/{k}$ which converges to infinity as $p$ goes to infinity. This similar property with our space $W^*_q(X)$ motivated us to study if $W^*_q(X) \cap \PSH(X,\omega)$ is contained in some $\mathcal{E}^{p(q)}$ with $p(q)$ increasing to infinity when $q$ goes to infinity. It turns out we can choose $p(q) = q-1$.

\vspace{0.2cm}
\begin{theorem} \label{MAglobal}
Let $(X,\omega)$ be a compact K\"{a}hler manifold. Then we have the following inclusions:
\begin{itemize}
    \item[(1)] $ W^*_q(X) \cap \PSH(X,\omega) \subset\mathcal{E}(X,\omega)$ for $q\geq 1$,
    \item[(2)]  $W^*_q(X)\cap \PSH(X,\omega) \subset \mathcal{E}^{q-1}(X,\omega)$ for $q\geq2$.
\end{itemize}
\end{theorem}
\vspace{0.2cm}

\noindent We prove this theorem by induction, relying on estimates around energy functionals. 

 In the local setting, the domain of definition of the Monge-Ampère operator $\mathcal{D}(\Omega)$ has been well understood after the works of Cegrell (\cite{Cegrell-MA}) and B\l ocki (\cite{Blocki-MA}). Recall that by \cite[Theorem 1.1]{Blocki-MA}, a function $\varphi$ belongs to $\mathcal{D}(\Omega)$ if $\varphi$ is a negative psh function on $\Omega$ and for every $z \in \Omega$, there exists an open neighborhood $U$ of $z$ in $\Omega$ and a sequence $\{\varphi_n\}$ of smooth negative psh functions such that the sequences 
 \begin{equation}\label{condlocalMA}
            |\varphi_n|^{k-p-2} d\varphi_n\wedge d^c\varphi_n \wedge (dd^c \varphi_n)^{p} \wedge (dd^c |z|^2)^{k-p-1}, \hspace{1cm} p = 0,\ldots,k-2,
        \end{equation}
are locally weakly bounded in $U$. Under this hypothesis, it is possible to define the Monge-Amp\`ere operator $(dd^c \varphi)^k$ in a satisfactory way (in the sense that we have the continuity of the Monge-Amp\`ere operator under decreasing psh approximations). It has been proved in \cite{Blocki-MA-dim2} that in $\mathbb{C}^2$, $W^{1,2}_{\local} \cap \PSH(\Omega) \subset \mathcal{D}(\Omega)$. Thus, the result holds for $W^*_{1,\local}$, where $W^*_{1,\local}$ is the set of functions that locally are functions in $W^*_1$. Our next theorem generalizes this fact by showing that the Monge-Amp\`ere operator $(dd^c\varphi)^k$ is well-defined in the sense of Cegrell-B\l ocki for psh locally $q$-complex Sobolev functions when $q\geq k-1$. 
\vspace{0.2cm}
\begin{theorem}\label{localMA}
    $ W^*_{q,\local}(\Omega)\cap \PSH(\Omega)\subset\mathcal{D}(\Omega)$ for $q\geq k-1$.
\end{theorem}
\vspace{0.2cm}

\noindent We prove this theorem by induction, relying on the condition \ref{condlocalMA}.
 %defining the product of $(1,1)$-currents using an induction on the number of currents.

%\vspace{0.2cm}
%The paper is organized as follows. In Section 2, we consider some basic properties of $W^*_q$ and corresponding Vigny's capacity. In Section 3, we construct a psh bound in order to prove Theorems~\ref{theoremA} and \ref{theoremA'}. Finally, in Section 4, we study the relationship between $W^*_q$ and the complex Monge-Amp\`{e}re equation and prove Theorems~\ref{MAglobal} and \ref{localMA}.

\vspace{0.2cm}

\noindent 
\textbf{Acknowledgments.} The first named author is supported by the MOE (Singapore) grant MOE-T2EP20120-0010. The second named author is supported by the Singapore International Graduate Award (SINGA). We would like to thank our teacher Tien-Cuong Dinh for useful discussions and constantly support during the preparation of this paper. In fact, the definition \ref{Def higher cplx Sobolev space} belongs to him. We also thank Qu\^oc Anh Ng\^o for  discussions on Sobolev spaces, Hoang-Son Tran for his comments on the first draft of this paper, and Gabriel Vigny for discussions concerning Section~\ref{sec 2}. Finally, we would like to thank the
anonymous referee for many suggestions that greatly improved the paper.

\section{Higher complex Sobolev space}\label{sec 2}
In this section, we prove some basic properties and give examples of our spaces. We also study their corresponding functional capacities.

\subsection{Quasinorms and compactness}

We now define some norms and quasinorms on $W^*_q$. First, we consider the local setting. Let $\Omega$ be a bounded domain in $\C^k$ and $\varphi \in W^*_q({\Omega}) $. Then by definition, there exist psh functions $\varphi_1,\ldots,\varphi_q$ in $\Omega$ such that 
    \begin{equation}\label{Defi Wq}
\|dd^c \varphi_j\| < \infty \text{ and }d\varphi_{j-1}\wedge \dc\varphi_{j-1}\leq\ddc\varphi_j \ \text{for all} \ j=1,\ldots,q.
    \end{equation}
Here we put $\varphi_0:= \varphi$. We call a sequence $(\varphi_1,\ldots,\varphi_q)$ satisfying (\ref{Defi Wq}) is a \textit{defining sequence} of $\varphi$ in $W^*_q(\Omega)$.
We define the \textit{$q$-star quasinorm} for $W_q^*(\Omega)$ as follows
\begin{equation}\label{def local quasinorm}\|\varphi\|_{W_q^*(\Omega)}:=\|\varphi\|_{L^2(\Omega)}+  \inf \Big\{ \sum_{j=1}^q \left \| dd^c \varphi_j \right \|^{1/2^j}  \Big\},\end{equation}
where the infimum is taken over all defining sequence $(\varphi_1,\ldots,\varphi_q)$ of $\varphi$. It is easy to see that one can replace $\inf$ by $\min$ in the definition \eqref{def local quasinorm}.

Next, we consider the global setting. Let $(X,\omega)$ be a compact K\"ahler manifold of dimension $k$ and $\varphi \in W^*_q({X})$. Then by definition, there exist qpsh functions $\varphi_1,\ldots,\varphi_q$ in $X$ and constants $c_1,\ldots,c_q\geq0$ such that
    \begin{equation}\label{Defi Wq X}
d\varphi_{j-1}\wedge \dc\varphi_{j-1}\leq c_j\omega+\ddc\varphi_j \ \text{for all} \ j=1,\ldots,q.
    \end{equation}    
Here we put $\varphi_0:= \varphi$. We call a sequence $((c_1,\varphi_1),\ldots,(c_q,\varphi_q))$  satisfying (\ref{Defi Wq X}) is a \textit{defining sequence} of $\varphi$ in $W^*_q(X)$. We define the \textit{$q$-star norm} for $W^*(X)$ as follows
\begin{equation}\label{def global norm}
\|\varphi\|_{W^*_q(X)}:=\|\varphi\|_{L^2(X)}+  \inf \Big\{ \sum_{j=1}^q c_j^{1/2^j} \Big\},\end{equation}
where the infimum is taken over all defining sequence $((c_1,\varphi_1),\ldots,(c_q,\varphi_q))$ of $\varphi$. As in the local case, one can replace $\inf$ by $\min$ in the definition \eqref{def global norm}. We note also that by Poincar\'{e}-Wirtinger inequality, if we replace $L^2$ in the definition of $\|\cdot\|_{W^*_q(X)}$ by $L^1$, we obtain an equivalent quasinorm. In the sequel, for simplicity, we will use the notation $\|\cdot\|_{*,q}$ instead of $\|\cdot\|_{W_q^*(\Omega)}$ and $\|\cdot\|_{W_q^*(X)}$ when there is no possible confusion of the domain.

\begin{proposition}\label{prop quasinorm local case}
    The function $\varphi\mapsto\|\varphi\|_{*,q}$ defines a quasinorm on $W^*_q$.
\end{proposition}

\begin{proof}
We begin with the local setting. We first check the homogeneity. Let $\psi = \lambda \varphi$, then $(\varphi_1,\ldots,\varphi_q)$ is a defining sequence for $\varphi$ iff $(|\lambda|^{2} \varphi_1,\ldots,|\lambda|^{2^q} \varphi_q)$ is a defining sequence for $\psi$. Hence, $\|\psi\|_{*,q} = |\lambda|\|\varphi\|_{*,q}$.

     We now only need to check the quasi-triangle inequality. Let $\varphi$ and $\psi$ be functions in $W^*_q(\Omega)$ where $(\varphi_1,\ldots,\varphi_q)$ and $(\psi_1,\ldots,\psi_q)$ are two corresponding defining sequences. Put $f=\varphi + \psi$. We claim that $(f_1,\ldots,f_q)$, with $f_j = \frac{4^{2^{j-1}}}{2}(\varphi_j+\psi_j)$, is a defining sequence of $f$ satisfying 
\[\sum_{j=1}^q \left \| dd^c f_j \right \|^{1/2^j} \leq 2^{1-1/2^q}\Big( \sum_{j=1}^q \left \| dd^c \varphi_j \right \|^{1/2^j}+\sum_{j=1}^q \left \| dd^c \psi_j \right \|^{1/2^j}\Big).\]

We prove it by induction on $q$. For $q=1$, we have
    \[df\wedge d^c f = d\varphi\wedge d^c \varphi+ d\psi\wedge d^c \psi + \left ( d\varphi\wedge d^c \psi+ d\psi\wedge d^c \varphi\right ).\]
    It follows from Cauchy-Schwarz inequality that
    \[ d\varphi\wedge d^c \psi+ d\psi\wedge d^c \varphi \leq   d\varphi\wedge d^c \varphi +  d\psi\wedge d^c \psi.\]
    This implies that $
        df\wedge d^c f \leq 2(d\varphi\wedge d^c \varphi +  d\psi\wedge d^c \psi) \leq  dd^c \left (2 \varphi_1 + 2 \psi_1 \right )$. 
Then for $f_1 = 2 (\varphi_1 +  \psi_1) $, we have $$df\wedge d^c f \leq  dd^c f_1 \text{ and }\|dd^c f_1\|^{1/2} \leq \sqrt{2}(\|dd^c \varphi_1\|^{1/2} + \|dd^c \psi_1\|^{1/2} ).$$
Hence, the proof for $q=1$ is complete.

Now, we assume that the desired property is true for $q-1$ where $q\geq 2$. We have $f_{q-1} = \frac{4^{2^{q-2}}}{2}(\varphi_{q-1}+\psi_{q-1}$). It thus follows from Cauchy-Schwarz inequality and the definition of defining sequences that
%by (\ref{induct1}), 
\[df_{q-1} \wedge d^c f_{q-1} \leq \frac{4^{2^{q-1}}}{2} dd^c \left ( \varphi_{q} +  \psi_{q} \right ).\]
Then for $f_q = \frac{4^{2^{q-1}}}{2}(\varphi_q+\psi_q)$, we have
\[df_{q-1} \wedge d^c f_{q-1} \leq \ddc f_q \text{ and } \|dd^c f_q\|^{1/2^q} \leq 2^{1-1/2^q}\left(\|dd^c \varphi_q\|^{1/2^q} + \|dd^c \psi_q\|^{1/2^q} \right).\]
Thus, by the induction hypothesis, we have
\begin{align*}
    \sum_{j=1}^q \left \| dd^c f_j \right \|^{1/2^j}
    &=\sum_{j=1}^{q-1} \left \| dd^c f_j \right \|^{1/2^j}+ \left \| dd^c f_q \right \|^{1/2^q}\\
 &\leq2^{1-1/2^{q-1}}\Big( \sum_{j=1}^{q-1} \left \| dd^c \varphi_j \right \|^{1/2^j}+\sum_{j=1}^{q-1} \left \| dd^c \psi_j \right \|^{1/2^j}\Big)\\&+ 2^{1-1/2^q}\Big(\|dd^c \varphi_q\|^{1/2^q} + \|dd^c \psi_q\|^{1/2^q} \Big)\\
&\leq2^{1-1/2^{q}}\Big( \sum_{j=1}^{q} \left \| dd^c \varphi_j \right \|^{1/2^j}+\sum_{j=1}^{q} \left \| dd^c \psi_j \right \|^{1/2^j}\Big),
\end{align*}
%Just by choosing $K(q) = \max\left\{K(q-1),\left(2c_{q-1}^2\right)^{1/2^{q}} \right\} 2^{1-\frac{1}{2^q}}$, 
as desired.

Next, we consider the global setting. The homogeneity is clear. Now let $\varphi$ and $\psi$ be functions in $W^*_q(X)$ where $((\alpha_1,\varphi_1),\ldots, (\alpha_q,\varphi_q))$ and $((\beta_1,\psi_1),\ldots,(\beta_q,\psi_q))$ are corresponding defining sequences. Let $f=\varphi+\psi$. As in the local case, it is sufficient to prove that $((\gamma_1,f_1),\ldots,(\gamma_q,f_q))$, with $\gamma_j = \frac{4^{2^{j-1}}}{2}(\alpha_j+\beta_j)$ and $ f_j = \frac{4^{2^{j-1}}}{2} (\varphi_j+\psi_j)$, is a defining sequence of $f$ satisfying
\[\sum_{j=1}^q \gamma_j^{1/2^j} \leq 2^{1-1/2^{q}} \Big(\sum_{j=1}^q \alpha_j^{1/2^j} + \sum_{j=1}^q \beta_j^{1/2^j}\Big).\]

We prove it by induction on $q$. For $q=1$, the same computation as in the local case gives
\begin{equation*}\label{df dcf global}
df\wedge d^c f \leq \left ( 2\alpha_1  +  2\beta_1\right ) \omega +  dd^c \left (2\varphi_1+2\psi_1\right ).
\end{equation*}
Then for $\gamma_1 = 2(\alpha_1 + \beta_1)$ and $f_1 = 2(\varphi_1+\psi_1)$, we have
\[df\wedge d^c f\leq \gamma_1\omega+\ddc f_1 \text{ and }\gamma_1^{1/2} \leq \sqrt{2}(\alpha_1^{1/2} + \beta_1^{1/2}).\]
This finishes the proof for $q=1$. Now we assume that the desired property is true for $q-1$ where $q\geq 2$. We have $\gamma_{q-1} = \frac{4^{2^{q-2}}}{2}(\alpha_{q-1}+\beta_{q-1}) $ and $f_{q-1} = \frac{4^{2^{q-2}}}{2}(\varphi_{q-1}+\psi_{q-1}$). It thus follows from Cauchy-Schwarz inequality and the definition of defining sequences that
\[df_{q-1} \wedge d^c f_{q-1} \leq \frac{4^{2^{q-1}}}{2}(\alpha_{q}+\beta_q)\omega+ \frac{4^{2^{q-1}}}{2} dd^c \left (\varphi_{q} + \psi_{q} \right ).\]
Then for $\gamma_q = \frac{4^{2^{q-1}}}{2}(\alpha_q+\beta_q)$ and $f_q = \frac{4^{2^{q-1}}}{2}(\varphi_q+\psi_q)$, we have
\[df_{q-1} \wedge d^c f_{q-1} \leq\gamma_q\omega+\ddc f_q \text{ and }
\gamma_q^{1/2^q} \leq 2^{1-1/2^{q}}\left(\alpha_q^{1/2^q} + \beta_q^{1/2^q} \right).\]
It thus follows from the induction hypothesis that
\begin{align*}
    \sum_{j=1}^q \gamma_j^{1/2^j} &=\sum_{j=1}^{q-1} \gamma_j^{1/2^j}+\gamma_q^{1/2^q}\\
    &\leq 2^{1-1/2^{q-1}} \Big(\sum_{j=1}^{q-1} \alpha_j^{1/2^j} + \sum_{j=1}^{q-1} \beta_j^{1/2^j}\Big)+ 2^{1-1/2^{q}}\Big(\alpha_q^{1/2^q} + \beta_q^{1/2^q} \Big)\\
    &\leq 2^{1-1/2^{q}} \Big(\sum_{j=1}^q \alpha_j^{1/2^j} + \sum_{j=1}^q \beta_j^{1/2^j}\Big),
\end{align*}
as desired.
The proof is complete.\end{proof}

We now consider a different norm on $W^*_q$. In the local setting, we define \begin{equation}\label{def local norm}\|\varphi\|_{**,q}:=\|\varphi\|_{L^2(\Omega)}+  \inf \Big\{ \left \| dd^c \varphi_q \right \|^{1/2^q}  \Big\},\end{equation}
where the infimum is taken over all defining sequence $(\varphi_1,\ldots,\varphi_q)$ of $\varphi$. In the global setting, we define
\begin{equation}\label{def norm global case}
    \|\varphi\|_{**,q} = \|\varphi\|_{L^2(X)} + \inf \{c_q^{1/2^q} \},
\end{equation} where the infimum is taken over all defining sequence $((c_1,\varphi_1),\ldots,(c_q,\varphi_q))$ of $\varphi$.

\begin{proposition}\label{prop norm case}
    The function $\varphi\mapsto\|\varphi\|_{**,q}$ defines a norm on $W^*_q$.
\end{proposition}

\begin{proof}
    We use the idea of \cite[Proposition 1]{Vigny}. We prove the local case only as the global case can be treated similarly (see the proof of Proposition~\ref{prop quasinorm local case}).

    The homogeneity is clear. We only need to check the triangle inequality. Let $\varphi$ and $\psi$ be functions in $W^*_q$ and $f:= \varphi+\psi$. Let $(\varphi_1,\ldots, \varphi_q)$ and $(\psi_1,\ldots,\psi_q)$ be defining sequences for $\varphi$ and $\psi$. Let $c$ be a positive number. We claim that $(f_1,\ldots,f_q)$ with
    \[ f_j:= (1+c)^{2^j-1}\varphi_j + (1+1/c)^{2^{j}-1} \psi_j,\]
    is a defining sequence for $f$.

    We prove this by induction on $q$. By Cauchy-Schwarz inequality, we have
    \begin{align*}df \wedge d^c f &\leq (1+c)d\varphi\wedge d^c  \varphi + (1+1/c) d\psi \wedge d^c \psi \\
&\leq  dd^c ((1+c)\varphi_1 + (1+1/c)\psi_1).
\end{align*} 
The claim follows for $q=1$. Assume that the claim is true for $q=1$ where $q\geq 2$. We have $ f_{q-1}:= (1+c)^{2^{q-1}-1}\varphi_{q-1} + (1+1/c)^{2^{q-1}-1} \psi_{q-1}$. Using Cauchy-Schwarz inequality again, we have
\begin{align*}df_{q-1} \wedge d^c f_{q-1} & \leq (1+c) (1+c)^{2^{q}-2} d\varphi_{q-1}\wedge d^c \varphi_{q-1} \\&  +(1+1/c)(1+1/c)^{2^{q}-2} d\psi_{q-1} \wedge d^c \psi_{q-1} \\
&\leq dd^c (  (1+c)^{2^q-1} \varphi_q + (1+1/c)^{2^q-1} \psi_q ).
\end{align*} 
The claim follows.

We now put $\|dd^c \varphi_q\| = \alpha_q, \|dd^c \psi_q\| = \beta_q$ and choose $c:= (\beta_q / \alpha_q)^{1/2^q}$. Then 
\begin{align*}
\|dd^c f_q\|&=
((1+c)^{2^q-1} \alpha_q + (1+1/c)^{2^q-1} \beta_q) \\  &\leq (\alpha_q^{1/2^q} + \beta_q^{1/2^q})^{2^q-1} \alpha_q^{1/2^q} + (\alpha_q^{1/2^q} + \beta_q^{1/2^q})^{2^q-1} \beta_q^{1/2^q} \\
&= (\alpha_q^{1/2^q} + \beta_q^{1/2^q})^{2^q}.
\end{align*}
Taking the power $2^{-q}$, we get $\|dd^c f_q\|^{1/2^q} \leq \|dd^c \varphi_q\|^{1/2^q} + \|dd^c \psi_q \|^{1/2^q}$ as desired.
\end{proof}

\begin{proposition}\label{prop equiv norm}
    Let $K$ be a relatively compact subset of $\Omega$. Then there exists a constant $C$ depending only on $K$ and $\Omega$ such that $C \|\varphi\|_{W^*_q(K)} \leq \|\varphi\|_{**,q} \leq \|\varphi\|_{*,q}$.
\end{proposition}

\begin{proof} The second inequality is clear. For the first inequality, let $\chi $ be a cut-off function with compact support in $\Omega$ such that $\chi \equiv 1$ on $K$. Let $\varphi \in W^*_q$ and $(\varphi_1,\ldots,\varphi_q)$ be a defining sequence that minimize \eqref{def local norm}. By Stokes' formula and Cauchy-Schwarz inequality, we have
\begin{align*} \int_K dd^c \varphi_j \wedge \omega^{k-1}  &\leq \int_\Omega \chi dd^c \varphi_j \wedge \omega^{k-1} = \Big|\int_\Omega d\chi \wedge d^c \varphi_j \wedge \omega^{k-1} \Big| \\
&\leq \Big(\int_\Omega d\chi \wedge d^c \chi \wedge \omega^{k-1}\Big)^{1/2} \cdot \Big(\int_\Omega d\varphi_j \wedge d^c \varphi_j \wedge \omega^{k-1}\Big)^{1/2}\\
&\leq C \Big(\int_\Omega d d^c \varphi_{j+1} \wedge \omega^{k-1}\Big)^{1/2},\end{align*}
for $1\leq j <q$. The result follows.
\end{proof}

\begin{remark}
It is natural to consider the quasinorm $\|\cdot\|_{*,q}$ since it takes information of all the elements in the defining sequence. Proposition~\ref{prop equiv norm} implies that if we only need to control the data in a compact subset, then we can construct a norm ($\|\cdot\|_{**,q}$) equivalent to the quasinorm $\|\cdot\|_{*,q}$. Moreover, Theorem~\ref{theoremA} is still true if we replace $\|\cdot\|_{*,q}$ by $\|\cdot\|_{**,q}$. However, it is not true in the global setting because one cannot bound $c_j$ by some constant times $c_{j+1}$ (this is a different point in the global setting). It is interesting to know if we can build a norm in $W^*_q$ equivalent to $\|\cdot\|_{*,q}$. When $q\geq 2$, since the mass of elements in the defining sequence can be chosen independently with each other (one can consider the $1$-dimensional case to see why), it seems to us that such a norm does not exist. 
%By Kolmogorov's normability criterion, one can try to prove this by showing that a convex set in $W^*_q$ is never bounded. But unfortunately, we do not know how to show that.
\end{remark}

\begin{remark}
Note that when the domain $\Omega$ is nice (e.g., convex domain), we have $W^*_1 = W^*$. This restriction is from the fact that we cannot solve the equation $dd^c \varphi = T$ in every domain. We also note that the spaces $W^*_q$ is decreasing in $q$.
\end{remark}

By using induction on $q$ and the idea of \cite[Proposition 4]{Vigny}, it is easy to prove that $W^*_q$ has the following  compactness property.

\begin{proposition}\label{compactness}
    Let $(\varphi_n)$ be a bounded sequence in $W^*_q$. Then there exists a subsequence $(\varphi_{n_j})$ and a function $\varphi \in W^*_{q}$ such that $\varphi_{n_j}$ converges weakly to $\varphi$ in $W^{1,2}$ and $\|\varphi\|_{*,q}\leq \lim\limits_{j\rightarrow \infty} \|\varphi_{n_j}\|_{*,q}$.
\end{proposition}

    Now, we prove that with the quasinorm $\|\cdot\|_{*,q}$, $W^*_q$ is a quasi-Banach space.
\begin{proposition}\label{quasi Banach}
    $W^*_q$ endowed with the quasinorm $\|\cdot\|_{*,q}$  is a quasi-Banach space.
\end{proposition}
\begin{proof}
     We prove by induction on $q$. For $q=1$, this is a slight modification of \cite[Proposition 1]{Vigny}. 
    Suppose that $W^*_{q-1}$ is a quasi-Banach space endowed with the quasinorm $\|\cdot\|_{*,q-1}$ for $q\geq 2$.
    Let $(\varphi_n)$ be a Cauchy sequence in $W^*_q$, then it is also a Cauchy sequence in $W^*_{q-1}$. Thus it converges to a function $\varphi \in W^*_{q-1}$. For every $\varepsilon>0$, there exists $N$ such that for $m,n > N$, we have $d(\varphi_n - \varphi_m)\wedge d^c(\varphi_n-\varphi_m) \leq dd^c \psi_{n,m}$,
    where $\psi_{n,m}$ is a function with $\|\psi_{n,m}\|_{*,q-1} < \varepsilon$. %Since $(\varphi_n - \varphi_m)_n$ converges in $W^{1,2}$ to $\varphi - \varphi_m$ as $n\to \infty$, $d(\varphi_n - \varphi_m)\wedge d^c(\varphi_n-\varphi_m) $ converges in $L^1$ to $d(\varphi - \varphi_m)\wedge d^c(\varphi-\varphi_m) $. 
    By Proposition~\ref{compactness}, we can find a subsequence of $(\psi_{n,m})_n$ which converges weakly in $W^{1,2}$ to a function $\psi_{m}$ such that $d(\varphi - \varphi_m)\wedge d^c(\varphi-\varphi_m)  \leq dd^c \psi_m$
    and $\|\psi_m \|_{*,q-1}\leq \varepsilon$. Thus, $\varphi \in W^*_q$ and $\varphi_n \to \varphi$ in $W^*_q$.
\end{proof}
\begin{remark}
    We can prove that $W^*_q$ endowed with the norm $\|\cdot\|_{**,q}$  is a Banach space by using similar arguments in the proof of Proposition \ref{quasi Banach}.
\end{remark}

\subsection{Examples} We now give specific examples for the functions in $W^*_q$.

\begin{example}
    Let $\varphi$ be a psh function in $\Omega$. Assume that $\varphi$ satisfies the following condition
    \begin{equation}\label{Donelly-Fefferman}
    d\varphi \wedge d^c \varphi \leq r dd^c \varphi\end{equation}
    for some $r>0$. Then it is clear by induction that $\varphi \in W^*_q(\Omega)$ for all $q$. Condition \eqref{Donelly-Fefferman} was introduced in \cite{donnellyFefferman-L2cohom-index-Bergman} and has many applications in studying the Bergman kernel of $\Omega$ (see, for example, \cite{berndtsson-charpentier-Sobolev-for-Bergman,blocki-suita}).
\end{example}

\begin{example}\label{example -log-varphi}
    Let $\varphi$ be an $\omega$-psh function in a compact K\"{a}hler manifold $(X,\omega)$. If $\varphi$ is bounded (assume that $0\leq \varphi \leq 1$), then we have
    \[d\varphi \wedge d^c \varphi = dd^c (\varphi^2)/2 - \varphi dd^c \varphi \leq \omega + dd^c(\varphi^2)/2,\]
    \[d(\varphi^2/2)\wedge d^c (\varphi^2/2) = \varphi^2 d\varphi \wedge d^c \varphi \leq d\varphi\wedge d^c\varphi \leq \omega + dd^c(\varphi^2/2).\]
    We can choose $((1,\varphi^2/2),\ldots,(1,\varphi^2/2))$ to be a defining sequence for $\varphi$. Hence $\varphi \in W^*_q(X)$ for all $q$ and we have the inclusion $\PSH(X,\omega) \cap L^\infty \subset \cap_q W^*_q(X)$.

     If $\varphi$ is unbounded (assume that $\varphi \leq -1$). Let $\psi = -\log(-\varphi)$. We have 
    \begin{equation*}
        d\psi \wedge d^c\psi = \frac{d\varphi \wedge d^c \varphi}{|\varphi|^2} \text{ and }
        dd^c \psi = -\frac{dd^c \varphi}{\varphi} + \frac{d\varphi \wedge d^c \varphi}{|\varphi|^2}\cdot
    \end{equation*}
    This implies that $d\psi \wedge d^c\psi = dd^c \psi + dd^c \varphi/\varphi \leq dd^c \psi + \omega$. We can choose $((1,\psi),\ldots,(1,\psi))$ to be a defining sequence for $\psi$. Therefore, $\psi \in W^*_q(X)$ for all $q$. See also \cite[Lemma 4.2]{DS_decay} and \cite[Example 1]{Vigny}.
\end{example}

\begin{example}\label{example max min}
    It has been observed in \cite[Section 2.3]{Vigny} (see also \cite[Proposition 4.1]{DS_decay}) that if $\varphi,\psi \in W^*_q$, then
    \begin{equation*}
        d\max(\varphi,\psi) \wedge d^c \max(\varphi,\psi) +
        d\min(\varphi,\psi) \wedge d^c \min(\varphi,\psi) \leq dd^c (\varphi_1+\psi_1),
    \end{equation*}
    where $\varphi_1,\psi_1$ are functions in $W^*_{q-1}$ such that $d\varphi \wedge d^c \varphi \leq dd^c \varphi_1$ and $d\psi\wedge d^c\psi \leq dd^c \psi_1$. So, both $\max(\varphi,\psi)$ and $\min(\varphi,\psi)$ belong to $W^*_q$ and 
    \[\|\max(\varphi,\psi)\|_{*,q} \leq 2\left(\|\varphi\|_{*,q} + \|\psi\|_{*,q}\right) \ \text{and} \ \|\min(\varphi,\psi)\|_{*,q} \leq 2\left(\|\varphi\|_{*,q} + \|\psi\|_{*,q}\right).\]
\end{example}

\begin{example}\label{mainexample} The following example is a higher version of \cite[Example 2]{Vigny}. Let $\varphi$ be the function defined by $-(-\log |z_1|^2)^\alpha$ in the unit ball $\mathbb{B}$ of $\mathbb{C}^k$. Then
    \[i\partial \varphi \wedge \bar{\partial} \varphi = \frac{idz_1\wedge d \bar{z_1}}{|z_1|^2(-\log |z_1|^2)^{2-2\alpha}}\cdot\]
    Let $\psi =-(-\log |z_1|^2)^{2\alpha}$. We have
    \[\bar{\partial} \psi = 2\alpha (-\log|z_1|^2 )^{2\alpha -1} \frac{1}{\bar{z_1}} d\bar{z_1},\qquad i\partial\bar{\partial} \psi = 2\alpha(1-2\alpha)\frac{i dz_1 \wedge d\bar{z_1}}{|z_1|^2 (-\log |z_1|^2)^{2-2\alpha}}\cdot\]
    This implies that
    \[i\partial \varphi \wedge \bar{\partial} \varphi = \frac{1}{2\alpha(1-2\alpha)}i\partial \bar{\partial} \psi\cdot\]
    By induction on $q$, we have $\varphi \in W^*_q(\B)$ if and only if $\alpha < 1/2^q$. This example will be used frequently to distinguish $W^*_q$ with different $q$ (see also Example~\ref{examMoser}).
\end{example}

\begin{example}\label{bmo for W_q} We consider the vanishing mean oscillation (VMO) property of functions in $W^*_q$. We highlight that this kind of property plays an important role in \cite{DNV-local-non-collapsing-bigclass}.

Let $u$ be a subharmonic function in $W^{1,2}(U)$ where $U$ is an open subset of $\mathbb{C}$. Then for any ball $B\subset U$, by Poincar\'{e}-Sobolev inequality, we have
\[\frac{1}{|B|} \int_B |u - m_B(u)| \leq c\Big(\int_B |\nabla u|^2\Big)^{1/2},\]
where $c$ is a constant not depending on $u$ and $B$. Letting the radius of $B$ goes to zero, we deduce that $\varphi$ is a VMO function. Thus, as pointed out in \cite[Theorem 1.1]{biard-Wu-equivalence-VMO-zeroLelong}, $u$ must have zero Lelong number at every point of $U$.

Let $\varphi $ be a psh function in $W^*_1(\Omega)$. By the slicing method (see \cite{DinhMarinescuVu}) and Siu's theorem (see \cite[Chapter. III (7.13)]{Demailly_ag}), $\varphi$ must have zero Lelong number at every point of $\Omega$. By the remark after \cite[Proposition 6]{Vigny}, we deduce that $W^*_q(\Omega)$ is a VMO space for every $q > 1$. The same result holds in the global setting. In fact, as we will see in Theorem~\ref{MAglobal}, if $\varphi \in W^*_1(X) \cap \PSH(X,\omega)$, then $\varphi \in \mathcal{E}(X,\omega)$. And it is well known that functions in $\mathcal{E}(X,\omega)$ have zero Lelong numbers at every point.
\end{example}

\subsection{Density theorems}
We now prove density theorems for $W^*_q$. The approximate sequences have been constructed in \cite{Vigny}, and we use the same construction for our spaces.

First, we consider the local case. We have the following result.
\begin{theorem}\label{localdense}

    Let $K$ be a relatively compact subset of $\Omega$ and $\varphi \in W^*_q(\Omega)$. Then there exists a sequence of smooth functions $(\varphi_n)$ converges to $\varphi$ in $W^{1,2}(K)$. Moreover, we have
$$\|\varphi\|_{W^*_q (K)} \leq \lim\limits_{n\rightarrow \infty} \|\varphi_n\|_{W^*_q (K)} \leq  \|\varphi\|_{W^*_q (\Omega)}.$$
\end{theorem}
\begin{proof}
    Take $\chi$ to be a non-negative smooth radial function with compact support in $\mathbb{C}^k$ such that $\int_{\mathbb{C}^k} \chi = 1$ and define $\chi_\varepsilon (z) = \varepsilon^{-2k} \chi (z/\varepsilon)$. Put $\varphi_\varepsilon = \varphi \ast \chi_\varepsilon$, then $\varphi_\varepsilon$ is well-defined in $K$ when $\varepsilon$ is small enough. Let $\psi \in W^*_{q-1}(\Omega) \cap \PSH(\Omega)$ such that $d \varphi \wedge d^c \varphi \leq dd^c \psi$. Let $(\varepsilon_n)_n$ be a sequence decreasing to zero. Define $\varphi_n  = \varphi_{\varepsilon_n}$, $\psi_n = \psi \ast \chi_{\varepsilon_n}$. By \cite[Lemma 5]{Vigny}, we have $d\varphi_n \wedge d^c \varphi_n \leq dd^c \psi_n$.

    Now, let $(\varphi_1,\ldots,\varphi_q)$ be a defining sequence for $\varphi$ and put $\varphi_{n,j}= \varphi_j \ast \chi_{\varepsilon_n}$ for $j=1,\ldots,q$. By induction, $(\varphi_{n,1},\ldots,\varphi_{n,q})$ is a defining sequence for $\varphi_n$. Moreover, we also have
    \begin{equation}\label{densitylocal}\lim_{n\rightarrow \infty} \int_{K} dd^c \varphi_{n,j} \wedge \omega^{k-1} \leq \int_{\Omega} dd^c \varphi_j \wedge \omega^{k-1} \text{ for } j= 1,\ldots,q.
    \end{equation}
    It follows that $(\varphi_n)$ is a bounded sequence in $W^*_q(K)$. Thus, by Proposition~\ref{compactness}, there exists a subsequence $(\varphi_{n_m})$ converges to $\varphi$ in $W^{1,2}(K)$ and satisfies the first inequality. By \eqref{densitylocal}, this sequence also satisfies the second inequality. The proof is complete.
\end{proof}

We now consider the case of compact K\"{a}hler manifolds. We have the following result.

\begin{theorem}
    Let $\varphi \in W^*_q(X)$. Then there exists a sequence of smooth functions $(\varphi_n)$ such that $\varphi_n$ converges to $\varphi$ weakly in $W^{1,2}(X)$. Moreover, there is a constant $c$ that does not depend on $\varphi$ such that $\displaystyle\lim\limits_{n\rightarrow \infty} \|\varphi_n\|_{*,q} \leq c\|\varphi\|_{*,q}$.
\end{theorem}

\begin{proof} The proof is similar to \cite[Theorem 10]{Vigny} (see also \cite[Theorem 2.2]{DNV-local-non-collapsing-bigclass} for a stronger version) so we only sketch it and refer the reader to cited papers for more details. Recall that by \cite{DS_regula}, there exist two sequences $(K_n^+), (K_n^-)$ of positive closed $(k,k)$ currents in $X\times X$ such that $K_n:=K_n^+ - K_n^-$ converges to the current of integration on the diagonal of $X\times X$. These currents are smooth outside the diagonal and $\|K_n^{\pm}(\cdot,y)\|_{L^1} \leq A$ where $A$ is a constant that depends only on $(X,\omega)$.

Consider $\varphi \in W^*_q(X)$ and put
\[\varphi_n(x) = \int_{y\in X} \varphi(y) K_n(x,y).\]
Let $\psi$ be the function in $W^*_{q-1}(X)$ such that $d\varphi \wedge d^c \varphi \leq c(\omega +dd^c \psi)$. Then we can bound $d\varphi_n \wedge d^c \varphi_n$ by the positive closed current
\[cA\int_{y\in X}  (K_n^{+ }(x,y) + K_n^-(x,y)) \wedge (\omega+dd^c \psi).\]
We can bound this current by $2cA^2 (\omega + dd^c \psi_n)$, where $\psi_n := \int_{y \in X} \psi(y) K_n^{\pm}(x,y)$. After iterating this convolution several times, we can make $\varphi_n$ smooth. The result follows by induction on $q$.
\end{proof}

%\begin{remark}
  %  Functions in $W^*_q$ and psh are stable under $\max$ operation. So it is natural to ask if one can use Richberg's regularization process to regularize functions in $W^*_q$. See \cite{BK07} for more information.
%\end{remark}

\subsection{Functional capacities}
Consider the case where $(X,\omega)$ is a compact K\"{a}hler manifold. Following \cite{Vigny}, for a Borel set $E$ in $X$, we define
\[L_q(E) = \left \{ \varphi \in W^*_q(X) : \varphi \leq -1 \ \text{a.e on some neighborhood of }E \text{ and } \varphi \leq 0 \ \text{on } X \right \}.\]
The corresponding functional capacity for $W^*_q$ can be defined as follows
\[\capa_q(E)=\inf \Big \{ \left \| \varphi \right \|_{*,q}^2 : \varphi \in L_q(E) \Big \}.\]
These capacities share similar properties with the functional capacity introduced by Vigny in \cite{Vigny}. We list below some important properties. The proofs are modifications of Vigny's proofs in \cite{Vigny}.

\begin{proposition}\label{propcap}
    The capacity $\capa_q$ satisfies the following properties
    \begin{itemize}
        
    \item[(1)] for $E\subset F \subset X$, $ \capa_q(E) \leq \capa_q (F)$;

    \item[(2)] if $(E_j)$ is a sequence of Borel sets in $X$, $\capa_q(\cup_j E_j) \leq  2\sum_j \capa_q(E_j)$;

    \item[(3)] $\capa_q(X) = 1$ and $\capa_q(E) \leq 1$ for any $E\subset X$;

    \item[(4)] if $(K_j)$ is a decreasing sequence of compact sets, $\displaystyle \lim\limits_{j\rightarrow \infty} \capa_q(K_j) = \capa_q(\cap_n K_j)$;

    \item[(5)] if $(E_j)$ is an increasing sequence of Borel sets, $\capa_q(\cup_j E_j) =\displaystyle  \lim\limits_{j\rightarrow \infty} \capa_q(E_j)$, that is, $\capa_q$ is a Choquet capacity.

    \end{itemize}
\end{proposition}

\begin{proof}
     We note that, due to Example~\ref{example max min}, in $W^*_q$, we have $$\left \| \min(\varphi,\psi) \right \|_{*,q}^2 \leq 8\Big(\left \| \varphi \right \|_{*,q}^2 + \left \| \psi \right \|_{*,q}^2\Big).$$ Moreover, $\|\mathbf{1}_X\|_{*,q} = 1$. We can now follow the proofs of \cite[Proposition 27 and Theorem 30]{Vigny} to prove this proposition.
\end{proof}

We now show that $\capa_q$ also characterizes pluripolar sets.

\begin{theorem}\label{equivcap}
    There exists a constant $B > 0$ such that for all Borel subset $E$ of $X$, we have 
    $$B^{-1}\capa_{\omega}(E)\leq \capa_q(E) \leq B(\capa_\omega(E))^{\frac{1}{k2^{q-1}}}.$$ In particular, $\capa_q(E) = 0$ if and only if $E$ is pluripolar.
\end{theorem}

 Recall that the capacity $\capa_{\omega}$ was defined by Ko\l odziej in \cite{Kolodziej_2003}. It is related to the well-known Bedford-Taylor capacity (\cite{Bedford_Taylor_82}), and is defined by
\[\capa_\omega(E) = \sup \Big \{ \int_E (\omega +dd^c u)^k :\ u \in \PSH(X,\omega),\ -1\leq u\leq 0 \Big \},\]
where $E$ is a Borel subset of $X$. We refer the reader to \cite{GZ} for more information on this capacity for local and global settings.

 We also need the following notion of capacity introduced by Dinh-Sibony in \cite{DS_tm}. It is related to the capacities of Alexander in \cite{Alexander} and of Siciak in \cite{Siciak-extremal}, see also \cite{Harvey_Lawson} and \cite{GZ}. For a Borel subset $E$ of $X$, we consider the function
\[V_{E}(x)= \sup \Big\{ u(x): u\in \PSH(X,\omega) \ \text{and} \
u \leq 0 \ \text{on} \ E \Big\}.\]
Then $V_{E}$ is a non-negative $\omega$-psh function. Define
\[\mathcal{J} (E)=\exp \Big ( -\sup_X V_{E}(x) \Big ).\]

Recall from \cite[Proposition 6.1]{GZ} the following relation between $\capa_\omega$ and $\mathcal{J}$.

\begin{lemma}\label{compareBTA}
    There is a constant $A>0$ such that for all compact subsets $K$ of $X$, we have
    \[\exp\Big ( -\frac{A}{\capa_\omega (K)} \Big )\leq \mathcal{J} (K)\leq e\cdot \exp\Big ( -\frac{1}{\capa_\omega (K)^{1/k}} \Big )\cdot\]
\end{lemma}

Now we can prove Theorem~\ref{equivcap}.

\begin{proof}[Proof of Theorem~\ref{equivcap}]
 Due to Example~\ref{example -log-varphi}, if $\varphi$ is a qpsh function such that $\varphi < -1$, then the function $\psi = -\log (-\varphi)$ belongs to $W^*_q$ for all $q$, and $\psi$ has the same poles set as $\varphi$. We can now follow the proofs of \cite[Proposition 28]{Vigny} to see the first inequality.

We now consider the second inequality. We closely follow the proof of \cite[Proposition 5.1]{dinh2022complex}. We only need to show this inequality for compact regular sets $K \subset X$ such that $1 \leq M < \infty$ where $M:= \sup_X V_K(x)$.

Define $f_K(x)= (V_{K}(x) - M)/M$. Then $f_K$ is equal $-1$ on $K$ with $-1\leq f_K \leq 0$ and $f_K$ is qpsh with $dd^c f_K + M^{-1}\omega \geq 0$. Since $f_K$ is a bounded qpsh function, it follows from Example~\ref{example -log-varphi} that $f_K\in W^*_q$ for all $q$.

We now compute $\|f_K\|_{*,q}$. 
Direct calculation gives us 
\begin{align*}
df_K\wedge d^c f_K &= -f_K dd^c f_K + \frac{1}{2} dd^c(f_K^2) 
\leq  \frac{\omega}{M}+\frac{1}{2} dd^c(f_K^2),
\end{align*}
\begin{align*} d(f_K^2/2)\wedge d^c (f_K^2/2) &= (f_K^2) df_K\wedge d^c f_K
\leq df_K\wedge d^c f_K 
\leq  \frac{\omega}{M}+\frac{1}{2} dd^c(f_K^2)\cdot
\end{align*}
Hence, we can choose $((M^{-1},f_K^2/2),\ldots,(M^{-1},f_K^2/2))$ as a defining sequence for $f_K$. Moreover, since $\max_X f_K = 0$, we have
\[\|f_K\|_{L^1} = \int_X -f_K \omega^n \leq \frac{A}{M}\]
for some constant $A$ depending only on $(X,\omega)$ because the set of $\omega$-psh function $u$ such that $\max_X u = 0$ is a compact subset in $L^1(X)$. Thus, we get
\begin{align*}\|f_K\|_{*,q} \leq \|f_K\|_{L^1} + \sum_{j=1}^q \frac{1}{M^{2^j}} \leq \frac{A}{M} + \sum_{j=1}^q \frac{1}{M^{1/2^j}}\leq \frac{B'}{M^{1/2^q}}
\end{align*}
for some constant $B'>0$ since $M \geq 1$.

So, by Lemma~\ref{compareBTA},
\[\capa_q(U) \leq \|f_K\|_{*,q}^2 \leq \frac{B'^2}{M^{1/2^{q-1}}} \leq B \left ( \capa_\omega(U) \right )^{\frac{1}{k2^{q-1}}}\cdot\]
The proof is complete.
\end{proof}

\begin{remark}
    Theorem~\ref{equivcap} directly shows that the capacities $\capa_q$ with $p\geq 1$, are equivalent capacities. Observe that the sequence $(\capa_q)_{q\geq1}$ is increasing and always bounded above by $1$. Given a Borel set $E$, it is an interesting question to study the behavior of $\capa_q(E)$ as $q$ goes to infinity and their relationship with the Hausdorff measures. 
\end{remark}

\begin{remark}
As in \cite[Remark 33]{Vigny}, one can define $\capa_q$ in the local case by the same method. It is also a Choquet capacity and characterizes pluripolar sets. 
\end{remark}

\section{Moser-Trudinger inequalities}\label{sec:3}

In this section, we will follow the strategy in \cite[Section 2]{Vigny-Vu-Lebesgue} to construct psh bound for functions in $W^*_q(\B)$ and prove Theorems \ref{theoremA}. The key point is that the additional information in the defining sequence will lead to stronger estimates which help us to improve the exponent in the Moser-Trudinger inequalities.

\subsection{Set up}

As pointed out in Example~\ref{bmo for W_q}, $W^*_q(\B)$ is a BMO space for every $q$. Therefore, by \cite{JN61-bmo}, for compact subset $K$ of $\B$, there are constants $c$ and $A$ such that $\int_K e^{c|\varphi|} \omega^k \leq A$ for all $\varphi$ in $W^*_q(X)$ with $\|\varphi\|_{*,q}\leq 1$ (this can also be seen as a consequence of \cite[Theorem 1.1]{DinhMarinescuVu}). We will use the following consequence of this fact.

\begin{lemma}\label{Cor of M-T DMV}
    Let $K$ be a compact subset of $\B$ and $m\in\N$. Then there exists a constant $c_1 > 0$ such that, for every $\varphi\in W^*_q(\B)$ with $\|\varphi\|_{*,q}\leq 1$,
    $$\int_K|\varphi|^m\omega^k\leq c_1, \text{ where } \omega = dd^c |z|^2.$$ 
\end{lemma}

Let $\varphi$ be a positive function in $ W^*_q(\B)$ with $\|\varphi\|_{*,q} \leq 1$ and $(\varphi_1,\ldots,\varphi_q)$ be a defining sequence for $\varphi$ such that
\[\|\varphi\|_{*,q} = \|\varphi\|_{L^2} + \sum_{j=1}^q \|dd^c \varphi_j\|^{1/2^j}.\]

\begin{remark}
        Recall that if $(\varphi_1,\ldots,\varphi_q)$ is a defining sequence for $\varphi$, then $(\varphi_{1,\varepsilon},\ldots,\varphi_{q,\varepsilon})$ is a defining sequence for $\varphi_\varepsilon$ where $\varphi_\varepsilon$ is the standard regularization of $\varphi$ and $\varphi_{j,\varepsilon}$ is the standard regularization of $\varphi_j$ for $j=1,\ldots,q$ (see Theorem~\ref{localdense}).
\end{remark}

We now provide some constructions similar to \cite{Vigny-Vu-Lebesgue}. The \textit{different point} is that we use data of the last term $\varphi_q$ in the defining sequence. Suppose that $\varphi_q \leq 0$. For $n>0$, we define $\phi_n=\max(\varphi_q,-n)$, $h_n=1+\phi_n/n$, $T_n=\ddc \left(h_n^2/2\right)$. Then $h_n \in\psh(\B)$, $0\leq h_n \leq 1$, and $h_n=0$ on $\{\varphi_q\leq-n\}$. Moreover, $T_n$ is a positive closed $(1,1)$-current that vanishes on the open set $\{\varphi_q<-n\}$.

We have the following lemma similar to \cite[Lemma 2.2]{Vigny-Vu-Lebesgue}.
\begin{lemma}\label{estimates}
    Assume that $\varphi_{q-1}$ and $\varphi_q$ are smooth. Then the following inequalities holds:
    \begin{itemize}
        \item[(1)] $dh_n\wedge\dc h_n\leq T_n$ and $h_n \ddc h_n\leq T_n$,
        \item[(2)] $d\varphi_{q-1}\wedge \dc\varphi_{q-1}\leq\ddc\phi_n$ on $\{h_n>0\}$,
        \item[(3)] $d\varphi_{q-1}\wedge \dc\varphi_{q-1}\wedge T_n\leq \ddc\phi_{n+1}\wedge T_n$.
        \item[(4)] $h_n d\varphi_{q-1}\wedge\dc\varphi_{q-1}\leq nh_n \ddc h_n\leq nT_n$.
    \end{itemize}
\end{lemma}

Similarly to \cite{Vigny-Vu-Lebesgue}, for every $m \in \N$, $K\Subset \B$, we define the following quantities:
$$I_{m,p,K}=\sup\limits_{v_1,\ldots,v_p}\int_K h_n^2\varphi^{2m}\ddc v_1\wedge\cdots\wedge\ddc v_p\wedge\omega^{k-p} \text{ for } 0\leq p \leq k,$$
$$J_{m,p,K}=\sup\limits_{v_1,\ldots,v_p}\int_K \varphi^{2m}T_n\ddc v_1\wedge\cdots\wedge\ddc v_p\wedge\omega^{k-p-1} \text{ for }\leq p \leq k-1.$$
Here, the supremum taken over all psh functions $v_j$ on $\B$ with $0\leq v_j \leq 1$. These quantities can be seen as a higher analogue of the (same notation) quantities in \cite{Vigny-Vu-Lebesgue}. As noted after Theorem~\ref{theoremA}, we use the canonical representative of $\varphi$ here, and thus all the integrals make sense. Note that all definitions depend on $n$. However, we omit the index $n$ as the estimates in the next section do not depend on $n$.

\subsection{Estimates}

We now prove some estimates for these quantities. As noted in the beginning of this section, the additional data in the defining sequence helps us to prove stronger estimates.

Let $K$ be a compact subset of $\B$. By standard regularization we can assume that $\varphi,\varphi_1,\ldots,\varphi_q$ are smooth functions in an open neighborhood of $K$. Let $\chi$ be a smooth cut-off function such that $\chi\equiv1$ on an open neighborhood of $K$, and $\chi$ is compactly supported on $B$. In what follows, we use $\lesssim$ or $\gtrsim$ to denote $\leq$ or $\geq$, respectively, modulo a multiplicative constant independent of $n$ and $\varphi$ provided $\|\varphi\|_{*,q}\leq 1$.

We first obtain the following lemma, which generalizes \cite[Lemma 2.3]{Vigny-Vu-Lebesgue}.

\begin{lemma}\label{m0K}
    There is a constant $c=c_{m,K}$ independent of $\varphi$ and $n$ such that $$J_{m,0,K}\leq c n^{\frac{m}{2^{q-1}}} \text{ for every } n.$$ 
\end{lemma}

\begin{proof}
It is sufficient to prove that
    $$J_m=\int_\B \chi^2\varphi^{2m}T_n\wedge\omega^{k-1}\lesssim n^{\frac{m}{2^{q-1}}}.$$

    It follows from the definition of $T_n$ and Stokes' formula that
    \begin{align*}
        J_m&=\frac{1}{2}\int_\B\chi^2\varphi^{2m}\ddc h_n^2\wedge\omega^{k-1}\\
        &=-\int_\B\chi\varphi^{2m} d\chi\wedge\dc h_n^2\wedge\omega^{k-1}-m\int_\B\chi^2\varphi^{2m-1} d\varphi\wedge\dc h_n^2\wedge\omega^{k-1}\\
        &=-2\Big(\int_\B\chi\varphi^{2m} h_n d\chi\wedge\dc h_n\wedge\omega^{k-1}+m\int_\B\chi^2\varphi^{2m-1} h_n d\varphi\wedge\dc h_n\wedge\omega^{k-1}\Big).
    \end{align*}

    Let $A_1$ and $A_2$ be respectively the first and second integrals inside the brackets. It follows from Cauchy-Schwarz inequality, Corollary~\ref{Cor of M-T DMV} and Lemma~\ref{estimates} that
    \begin{align*}
        A_1^2&\leq \Big(\int_\B\chi^2\varphi^{2m}dh_n\wedge \dc h_n\wedge\omega^{k-1}\Big) \Big(\int_\B \varphi^{2m} h_n^2d\chi\wedge\dc\chi\wedge\omega^{k-1}\Big)\\
        &\lesssim\Big(\int_\B \chi^2\varphi^{2m} T_n\wedge\omega^{k-1}\Big) \Big(\int_{\supp(\chi)}\varphi^{2m}\wedge\omega^{k} \Big)\\
        &\lesssim J_m,
    \end{align*}
    and 
    \begin{align*}
        A_2^2&\leq \Big( \int_\B\chi^2\varphi^{2m}dh_n\wedge \dc h_n\wedge\omega^{k-1} \Big) \Big(\int_\B \chi^2\varphi^{2m-2} h_n^2d\varphi\wedge\dc\varphi\wedge\omega^{k-1} \Big)\\
        &\leq \Big(\int_\B \chi^2\varphi^{2m} T_n\wedge\omega^{k-1} \Big) \Big( \int_\B \chi^2\varphi^{2m-2} h_n^2\ddc\varphi_1\wedge\omega^{k-1} \Big)\\
        &\leq J_m B_{1,m-1},
    \end{align*}
    where $$B_{j,m}=\int_\B \chi^2\varphi^{2m} h_n^2\ddc\varphi_j\wedge\omega^{k-1} \text{ for } j =1 ,\ldots,q. $$

    This implies $J_m\lesssim \sqrt{J_m}+\sqrt{J_m}\sqrt{B_{1,m-1}},$ and hence
    \begin{equation}\label{Jm Bm-1}
        J_m\lesssim B_{1,m-1}.
    \end{equation}

     To estimate $B_{1,m-1}$, we first observe that, by Stokes' formula, for $j=1, \ldots ,q-1$,
    \begin{align*}
        B_{j,m}&=-2\Big( \int_\B\chi\varphi^{2m}h_n^2d\chi\wedge\dc\varphi_j\wedge\omega^{k-1}\\
        &+ m\int_\B\chi^2\varphi^{2m-1}h_n^2d\varphi\wedge\dc\varphi_j\wedge\omega^{k-1}+\int_\B\chi^2\varphi^{2m}h_ndh_n\wedge\dc\varphi_j\wedge\omega^{k-1} \Big).
    \end{align*}

Let $C_{j,1},C_{j,2},$ and $ C_{j,3}$ be the first, second, and third integrals inside the brackets respectively. It follows from Cauchy-Schwarz inequality that
\begin{align*}
    C_{j,1}^2&\leq \Big(\int_\B \varphi^{2m} h_n^2 d\chi\wedge\dc\chi\wedge\omega^{k-1} \Big)
\Big(\int_\B\chi^2\varphi^{2m}h_n^2d\varphi_j\wedge\dc\varphi_j\wedge\omega^{k-1} \Big)\\
    &\lesssim \Big(\int_{\supp (\chi)}\varphi^{2m}\omega^k
    \Big) \Big(\int_\B\chi^2\varphi^{2m}h_n^2d\varphi_j\wedge\dc\varphi_j\wedge\omega^{k-1}\Big),
\end{align*}
\begin{align*}
    C_{j,2}^2&\leq \Big(\int_\B\chi^2 \varphi^{2m-2}h_n^2d\varphi\wedge\dc\varphi\wedge\omega^{k-1} \Big)\Big(\int_\B\chi^2\varphi^{2m}h_n^2d\varphi_j\wedge\dc\varphi_j\wedge\omega^{k-1} \Big)\\
    &\leq \Big( \int_\B\chi^2 \varphi^{2m-2}h_n^2\ddc\varphi_1\wedge\omega^{k-1} \Big) \Big(\int_\B\chi^2\varphi^{2m}h_n^2d\varphi_j\wedge\dc\varphi_j\wedge\omega^{k-1}\Big)
\end{align*}
and 
\begin{align*}
    C_{j,3}^2&\leq \Big( \int_\B\chi^2\varphi^{2m}dh_n\wedge\dc h_n\wedge\omega^{k-1} \Big) \Big(\int_\B\chi^2\varphi^{2m}h_n^2d\varphi_j\wedge\dc\varphi_j\wedge\omega^{k-1}\Big).
\end{align*}
It thus follows from the definition of the defining sequence, Corollary~\ref{Cor of M-T DMV}, Lemma~\ref{estimates}, and inequality \eqref{Jm Bm-1} that
$$ C_{j,1}^2\lesssim\begin{cases}
    B_{j+1,m} &\text{ if }j<q-1\\
    n B_{1,m-1} &\text{ if }j=q-1,
\end{cases}\qquad 
C_{j,2}^2\lesssim\begin{cases}
    B_{1,m-1}B_{j+1,m} &\text{ if }j<q-1\\
    n B_{1,m-1}^2 &\text{ if }j=q-1,
\end{cases}$$
 $$C_{j,3}^2\lesssim\begin{cases}
      B_{1,m-1}B_{j+1,m} &\text{ if }j<q-1\\
      n B_{1,m-1}^2 &\text{ if }j=q-1.
 \end{cases}$$
 Note that we use \eqref{Jm Bm-1} and the fact that $0\leq h_n\leq 1$ when $j=q-1$.

Then we have
$$B_{j,m}^2\lesssim\begin{cases}
    B_{j+1,m}+B_{1,m-1}B_{j+1,m} &\text{ if }j<q-1\\
    n B_{1,m-1}^2 &\text{ if }j=q-1,
\end{cases}$$
which implies
$$\begin{cases}
  B_{j,m}^2\lesssim B_{1,m-1}B_{j+1,m}\text{ for }j<q-1\\
    B_{q-1,m}\lesssim \sqrt{n} B_{1,m-1},
\end{cases}$$
and hence,
$$B_{1,m}^{2^{q-2}}\lesssim B_{q-1,m} B_{1,m-1}^{2^{q-2}-1}\lesssim \sqrt{n} B_{1,m-1}^{2^{q-2}}. $$
It thus follows that
$$B_{1,m}\lesssim n^{\frac{1}{2^{q-1}}} B_{1,m-1},$$ and hence
\begin{equation}\label{B1m nm2p-1}
    B_{1,m}\lesssim n^{\frac{m}{2^{q-1}}}.
\end{equation}

Now, combining inequalities \eqref{Jm Bm-1} and \eqref{B1m nm2p-1} gives us
$$J_m\lesssim n^{\frac{m}{2^{q-1}}},$$
as desired.
\end{proof}

The following lemma generalizes \cite[Lemma 2.4]{Vigny-Vu-Lebesgue}.

\begin{lemma}\label{mpK}
    There is a constant $c=c_{m,K}$ independent of $\varphi$ and $n$ such that $$J_{m,p,K}\leq c n^{\frac{m}{2^{q-1}}} \text{ for every } n.$$ 
\end{lemma}

\begin{proof}
We prove the lemma by induction in $p$. If $p=0$, the desired assertion is Lemma~\ref{m0K}. Assume now that it is true for all $p'$ with $p' \leq p-1$. It is sufficient to prove that
$$J_{m,p}=\sup\limits_{v_1,\ldots,v_p}\int_\B \chi^2 \varphi^{2m}T_n\ddc v_1\wedge\cdots\wedge\ddc v_p\wedge\omega^{k-p-1}\lesssim n^{\frac{m}{2^{q-1}}},$$
where the supremum taken over all psh functions $v_j$ on $\B$ with $0 \leq v_j \leq 1$.

We prove this inequality by induction on $m$ ($p$ now fixed). When $m=0$, it is obvious. Assume that it is true for all $m'$ with $m'\leq m-1$. Let $v_1,\ldots,v_p$ be psh functions in $\B$ that take values in $[0,1]$. We set $R=\ddc v_2\wedge\cdots\wedge \ddc v_p\wedge T_n\wedge \omega^{k-p-1}.$ It follows from Stokes' formula that
\begin{align*}
    \int_\B &\chi^2 \varphi^{2m}T_n\ddc v_1\wedge\cdots\wedge\ddc v_p\wedge\omega^{k-p-1}\\
&=-2\Big(\int_\B\chi\varphi^{2m}d\chi\wedge\dc v_1\wedge R+m\int_\B\chi^2\varphi^{2m-1}d\varphi\wedge\dc v_1\wedge R\Big).
\end{align*}

Let $D_1$ and $D_2$ be the first and second integrals inside the brackets respectively. By the Cauchy-Schwarz inequality, the induction hypothesis on $p$, the induction hypothesis on $m$, and Lemma~\ref{estimates}, we have
\begin{align*}
    D_1^2&\leq \Big(\int_\B \varphi^{2m} d\chi\wedge\dc\chi\wedge R \Big) \Big( \int_\B\chi^2\varphi^{2m} dv_1\wedge\dc v_1\wedge R \Big)\\
    &\lesssim \Big(\int_{\supp(\chi)}\varphi^{2m}R\wedge\omega \Big) \Big(
    \int_\B\chi^2\varphi^{2m}\ddc(v_1^2)\wedge R \Big)\\
    &\lesssim J_{m,p-1,\supp (\chi)} J_{m,p}\\
    &\lesssim n^{\frac{m}{2^{q-1}}} J_{m,p},
\end{align*}
and
\begin{align*}
    D_2^2&\leq \Big(\int_\B\chi^2\varphi^{2m-2}d\varphi\wedge\dc\varphi\wedge R \Big) \Big( \int_\B \chi^2\varphi^{2m}dv_1\wedge\dc v_1\wedge R \Big)\\
    &\lesssim \Big(\int_\B\chi^2\varphi^{2m-2}\ddc\varphi_1\wedge R \Big) \Big(\int_\B\chi^2\varphi^{2m}\ddc (v_1^2)\wedge R \Big)\\
    &\lesssim E_{1,m-1} J_{m,p},
\end{align*}
where
$$E_{j,m}=\sup_{v_2,\ldots,v_p}\int_\B \chi^2\varphi^{2m}\ddc\varphi_j\wedge R \text{ for } j=1,\ldots,q.$$
Note that since $\chi$ depends on $K$, the estimate here only depends on $K$.

Then by taking the supremum over all such $v_1,\ldots,v_p$, we observe that
\begin{equation}\label{Jmp Em-1}
    J_{m,p}\lesssim n^{\frac{m}{2^{q-1}}}+E_{1,m-1}.
\end{equation}

To estimate $E_{1,m-1}$, we consider $v_1,\ldots,v_p$ as above.
It follows from Stokes' formula that, for $j=1,\ldots,q-1$,
$$\int_\B \chi^2\varphi^{2m}\ddc\varphi_j\wedge R=-2\Big( \int_\B\chi\varphi^{2m}d\chi\wedge\dc\varphi_j\wedge R+m\int_\B\chi^2\varphi^{2m-1}d\varphi\wedge\dc\varphi_j\wedge R \Big).$$

Let $F_{j,1}$ and $F_{j,2}$ be the first and second integrals inside the brackets respectively. By the Cauchy-Schwarz inequality, the induction hypothesis on $p$, and the induction hypothesis on $m$, we have
\begin{align*}
    F_{j,1}^2&\leq \Big(\int_\B\varphi^{2m} d\chi\wedge\dc\chi\wedge R \Big) \Big(\int_\B\chi^2\varphi^{2m}d\varphi_j\wedge d^c\varphi_j\wedge R \Big)\\
    &\lesssim  J_{m,p-1,\supp(\chi)}\Big(\int_\B\chi^2\varphi^{2m}d\varphi_j\wedge d^c\varphi_j\wedge R \Big)\\
    &\lesssim n^{\frac{m}{2^{q-1}}} \Big(\int_\B\chi^2\varphi^{2m}d\varphi_j\wedge d^c\varphi_j\wedge R\Big),
\end{align*}
and
\begin{align*}
    F_{j,2}^2&\leq \Big(\int_\B\chi^2\varphi^{2m-2}d\varphi\wedge \dc\varphi\wedge R \Big) \Big(\int_\B\chi^2\varphi^{2m}d\varphi_j\wedge\dc\varphi_j\wedge R \Big)\\
    &\leq \Big(\int_\B\chi^2\varphi^{2m-2} \ddc\varphi_1\wedge R \Big) \Big(\int_\B\chi^2\varphi^{2m}d\varphi_j\wedge\dc\varphi_j\wedge R \Big)\\
    &\leq  E_{1,m-1} \Big(\int_\B\chi^2\varphi^{2m}d\varphi_j\wedge\dc\varphi_j\wedge R\Big).
\end{align*}

Taking the supremum over all such $v_1,\ldots,v_p$, using Lemma~\ref{estimates} and inequality \eqref{Jmp Em-1}, we have
$$E_{j,m}^2\lesssim\begin{cases}
    E_{j+1,m}\Big(E_{1,m-1}+ n^{\frac{m}{2^{q-1}}} \Big) &\text{ if }j<q-1\\
    n\Big( n^{\frac{m}{2^{q-1}}} (n^{\frac{m}{2^{q-1}}}+E_{1,m-1}) + E_{1,m-1}(n^{\frac{m}{2^{q-1}}}+E_{1,m-1}) \Big)&\text{ if }j=q-1.
\end{cases}$$
It thus follows that
$$\begin{cases}
    E_{j,m}^2\lesssim E_{j+1,m} \Big(E_{1,m-1}+ n^{\frac{m}{2^{q-1}}}\Big) \text{ if }j<q-1\\
    E_{q-1,m}\lesssim \sqrt{n} \Big(E_{1,m-1}+ n^{\frac{m}{2^{q-1}}}\Big).
\end{cases}$$

Then we have
$$E_{1,m}^{2^{q-2}}\leq E_{q-1,m} \Big(E_{1,m-1}+ n^{\frac{m}{2^{q-1}}}\Big)^{2^{q-2}-1}\lesssim \sqrt{n} \Big(E_{1,m-1}+ n^{\frac{m}{2^{q-1}}}\Big)^{2^{q-2}},$$
and hence,
$$E_{1,m}\lesssim n^{\frac{1}{2^{q-1}}}\Big(E_{1,m-1}+ n^{\frac{m}{2^{p-1}}}\Big).$$
Therefore, we obtain
\begin{equation}\label{E1m nm2q-1}
    E_{1,m}\lesssim n^{\frac{m+1}{2^{q-1}}}.
\end{equation}

Combining inequalities \eqref{Jmp Em-1} and \eqref{E1m nm2q-1} gives us $$J_{m,p}\lesssim n^{\frac{m}{2^{q-1}}}$$ as desired.
\end{proof}

Finally, we estimate $I_{m,p,K}$, these estimates generalize \cite[Lemma 2.5]{Vigny-Vu-Lebesgue}.

\begin{lemma}\label{ImpK}
    There is a constant $c=c(m,K)$ independent of $\varphi$ and $n$ such that $$I_{m,p,K}\leq c n^{\frac{m}{2^{q-1}}}$$ for every $n$.
\end{lemma}

\begin{proof}
    We argue similarly as in the proof of Lemma~\ref{mpK}, by induction on $p$. If $p = 0$, the desired assertion follows from Corollary~\ref{Cor of M-T DMV}. We assume now that it is true for every $p'$ with $p'\leq p-1$. To prove the desired assertion, it suffices to prove that
    $$I_{m,p}=\sup\limits_{v_1,\ldots,v_p}\int_\B\chi^2 \varphi^{2m} h_n^2\ddc v_1\wedge\cdots\wedge\ddc v_p\wedge\omega^{k-p}\lesssim n^{\frac{m}{2^{q-1}}},$$
    where the supremum taken over all psh functions $v_j$ on $\B$ with $0 \leq v_j \leq 1$. 
    
    We prove this inequality by induction on $m$ ($p$ now fixed). When $m=0$, it is obvious. Assume that it is true for all $m'$ with $m'\leq m-1$. Let $v_1,\ldots,v_p$ be psh functions take value in $[0,1]$. We set $R'=\ddc v_2\wedge\cdots\wedge \ddc v_p\wedge \omega^{k-p}$. It follows from Stokes' formula that
    \begin{align*}
        \int_\B\chi^2  \varphi^{2m} h_n^2\ddc v_1\wedge R'=&-2\Big( 
        \int_\B \chi\varphi^{2m} h_n^2 d\chi\wedge\dc v_1\wedge R'\\ 
        +&\int_\B \chi^2\varphi^{2m}h_n dh_n\wedge d^c v_1\wedge R'+m\int_\B \chi^2\varphi^{2m-1} h_n^2 d\varphi\wedge \dc v_1\wedge R'
        \Big).
    \end{align*}

    Let $G_1, G_2$, and $ G_3$ be the first, second, and third integrals inside the brackets respectively. By the Cauchy-Schwarz inequality, the induction hypothesis on $p$, the induction hypothesis on $m$, and Lemma~\ref{mpK}, we have
    \begin{align*}
        G_1^2&\leq \Big(\int_\B h_n^2\varphi^{2m}d\chi\wedge\dc\chi\wedge R' \Big) \Big(\int_\B h_n^2\varphi^{2m}\chi^2 dv_1\wedge \dc v_1\wedge R' \Big)\\
        &\lesssim I_{m,p-1,\supp(\chi)} \Big(\int_\B h_n^2\varphi^{2m}\chi^2 \ddc(v_1^2)\wedge R' \Big)\\
        &\lesssim n^{\frac{m}{2^{q-1}}} I_{m,p},
    \end{align*}
   \begin{align*}
        G_2^2&\leq \Big(\int_\B \chi^2 \varphi^{2m} dh_n\wedge d^c h_n\wedge R'\Big) \Big( \int_\B\chi^2 \varphi^{2m} h_n^2 dv_1\wedge\dc v_1\wedge R' \Big)\\
        &\leq \Big(\int_\B\chi^2\varphi^{2m} T_n\wedge R'\Big) \Big( \int_\B \chi^2\varphi^{2m} h_n^2 \ddc(v_1^2)\wedge R' \Big)\\
        &\leq J_{m,p-1,\supp(\chi)} I_{m,p}\\
        &\lesssim n^{\frac{m}{2^{q-1}}} I_{m,p},
   \end{align*} 
   and
   \begin{align*}
       G_3^2&\leq  \Big(\int_\B \chi^2\varphi^{2m-2}h_n^2 d\varphi\wedge\dc\varphi\wedge R' \Big) \Big( \int_\B\chi^2\varphi^{2m}h_n^2 dv_1\wedge \dc v_1\wedge R' \Big)\\
       &\leq \Big(\int_\B \chi^2\varphi^{2m-2}h_n^2 \ddc\varphi_1\wedge R'\Big) \Big( \int_\B\chi^2\varphi^{2m}h_n^2 \ddc(v_1^2)\wedge R' \Big)\\
       &\leq H_{1,m-1} I_{m,p}
    \end{align*}
    where
    $$H_{j,m}=\sup\limits_{v_2,\ldots,v_p} \int_\B \chi^2\varphi^{2m}h_n^2 \ddc\varphi_j\wedge R' \text{ for } j =1,\ldots,q.$$

    Then, by taking the supremum over all sucht $v_1,\ldots,v_p$, we have
    \begin{equation}\label{Imp nm/2 Hm-1}
        I_{m,p}\lesssim n^{\frac{m}{2^{q-1}}}+H_{1,m-1}.
    \end{equation}

    To estimate $H_{1,m-1}$, we consider $v_2,\ldots,v_p$ as above.
It follows from Stokes' formula that, for $j=1,\ldots,q-1$,
\begin{align*}
        \int_\B \chi^2\varphi^{2m}h_n^2 \ddc\varphi_j\wedge R'&=-2\Big( \int_\B\chi\varphi^{2m}h_n^2d\chi\wedge\dc\varphi_j\wedge R'+\\
        &m\int_\B\chi^2\varphi^{2m-1}h_n^2d\varphi\wedge\dc\varphi_j\wedge R' + \int_\B\chi^2\varphi^{2m}h_ndh_n\wedge\dc\varphi_j\wedge R'\Big).
    \end{align*}

Let $L_{j,1}, L_{j,2}$, and $ L_{j,3}$ be the first, second, and third integrals inside the brackets respectively. It follows from Lemma~\ref{estimates}, Lemma~\ref{mpK}, Cauchy-Schwarz inequality and the induction hypothesis on $p$ that
\begin{align*}
    L_{j,1}^2&\leq \Big(\int_\B \varphi^{2m} h_n^2 d\chi\wedge\dc\chi\wedge R' \Big) \Big(\int_\B\chi^2\varphi^{2m}h_n^2d\varphi_j\wedge\dc\varphi_j\wedge R' \Big)\\
    &\leq \Big(\int_{\supp(\chi)}\varphi^{2m}h_n^2\wedge R'\wedge\omega \Big) \Big(\int_\B\chi^2\varphi^{2m}h_n^2d\varphi_j\wedge\dc\varphi_j\wedge R'\Big)\\
    &\leq I_{m,p-1,\supp(\chi)} \Big(\int_\B\chi^2\varphi^{2m}h_n^2d\varphi_j\wedge\dc\varphi_j\wedge R' \Big)\\
    &\lesssim n^{\frac{m}{2^{q-1}}} \Big(\int_\B\chi^2\varphi^{2m}h_n^2d\varphi_j\wedge\dc\varphi_j\wedge R'\Big),
\end{align*}
\begin{align*}
    L_{j,2}^2&\leq \Big(\int_\B\chi^2 \varphi^{2m-2}h_n^2d\varphi\wedge\dc\varphi\wedge R' \Big) \Big(\int_\B\chi^2\varphi^{2m}h_n^2d\varphi_j\wedge\dc\varphi_j\wedge R' \Big)\\
    &\leq \Big(\int_\B\chi^2 \varphi^{2m-2}h_n^2\ddc\varphi_1\wedge R'\Big) \Big(\int_\B\chi^2\varphi^{2m}h_n^2d\varphi_j\wedge\dc\varphi_j\wedge R' \Big)\\
    &\lesssim H_{1,m-1}
    \Big(\int_\B\chi^2\varphi^{2m}h_n^2d\varphi_j\wedge\dc\varphi_j\wedge R'\Big),
\end{align*}
and
\begin{align*}
    L_{j,3}^2&\leq \Big(\int_\B\chi^2\varphi^{2m}dh_n\wedge\dc h_n\wedge R'\Big) \Big(
    \int_\B\chi^2\varphi^{2m}h_n^2d\varphi_j\wedge\dc\varphi_j\wedge R' \Big)\\
    &\leq \Big(\int_\B\chi^2\varphi^{2m} T_n\wedge R' \Big) \Big(\int_\B\chi^2\varphi^{2m}h_n^2d\varphi_j\wedge\dc\varphi_j\wedge R' \Big)\\
    &\leq J_{m,p-1,\supp(\chi)} \Big(\int_\B\chi^2\varphi^{2m}h_n^2d\varphi_j\wedge\dc\varphi_j\wedge R' \Big)\\
    &\lesssim n^{\frac{m}{2^{q-1}}} \Big(\int_\B\chi^2\varphi^{2m}h_n^2d\varphi_j\wedge\dc\varphi_j\wedge R'\Big).
\end{align*}

Taking the supremum over all such $v_2,\ldots,v_p$, using Lemma~\ref{estimates}, Lemma~\ref{mpK}, and inequality \eqref{Imp nm/2 Hm-1}, we have
$$H_{j,m}^2\lesssim\begin{cases}
    H_{j+1,m}\Big(H_{1,m-1}+ n^{\frac{m}{2^{q-1}}}\Big)\text{ if } j<q-1\\
    n \Big(H_{1,m-1}+n^{\frac{m}{2^{q-1}}} \Big)^2\text{ if } j=q-1.
\end{cases}$$
It thus follows that
$$\begin{cases}
    H_{k,m}^2\lesssim H_{k+1,m}\Big(H_{1,m-1}+ n^{\frac{m}{2^{q-1}}}\Big)\text{ if } k<q-1\\
   H_{q-1,m}\lesssim \sqrt{n} \Big(H_{1,m-1}+n^{\frac{m}{2^{q-1}}} \Big),
\end{cases}$$
and hence
$$H_{1,m-1}^{2^{q-2}}\lesssim H_{q-1,m}\Big(H_{1,m-1}+n^{\frac{m}{2^{q-1}}} \Big)^{2^{q-2}-1}\lesssim \sqrt{n} \Big(H_{1,m-1}+n^{\frac{m}{2^{q-1}}} \Big)^{2^{q-2}}.$$

Then we have
$$H_{1,m}\lesssim n^{\frac{1}{2^{q-1}}}\Big(H_{1,m-1}+n^{\frac{m}{2^{q-1}}} \Big),$$
which implies
\begin{equation}\label{H1m-1 m2q-1}
    H_{1,m}\lesssim n^{\frac{m+1}{2^{q-1}}}.
\end{equation}

Combining inequalities (\ref{Imp nm/2 Hm-1}) and (\ref{H1m-1 m2q-1}) gives
$$I_{m,p}\lesssim n^{\frac{m}{2^{q-1}}},$$ 
as desired.
\end{proof}

\subsection{Bounding by plurisubharmonic functions} With all these estimates in hand, we can now construct a psh bound for functions in $W^*_q(\B)$.

\begin{theorem}\label{main result}
    Let $\varphi \in W^*_q(\B)$ with $\|\varphi\|_{*,q}\leq 1$ and $\alpha \in[1,2^{q})$. Then for every compact subset $K$ of $\B$, there exist a constant $C > 0$ and a psh function $u$ on $\B$ such that
$$|\varphi|^\alpha \leq -u \text{ on } K \text{ and }\|u\|_{L^1(K)} \leq C.$$
\end{theorem}

\begin{proof}
    The proof of this theorem is almost line by line with the proof of \cite[Theorem 1.3]{Vigny-Vu-Lebesgue} (just change the coefficient form $1$ to $q$). We only present the construction and refer the reader to the cited paper for details of the proof.

    Write $\varphi = \max(\varphi,0) + \min(\varphi,0)$, we can assume that $\varphi \geq 0$. Fix $\alpha \in [1,2^q)$ and choose a constant $ \lambda \in (\max (2^\alpha,2^{2^{q-1}}),2^{2^{q}})$. For $n\in\N$, we set
$$K_n=\Big\{z\in K:\ \varphi(z)\geq 2^n,\ \varphi_q\geq -\lambda^n\Big\}.$$
Define the extremal functions
\[u_n = u_{K_n}:= \sup\Big\{u\in\psh(\B):\ u\leq 0 \text{ on }\B,\ u\leq-1\text{ on } K_n\Big\}.\]

Let $(\varphi_1,\ldots,\varphi_q)$ be a defining sequence for $\varphi$ such that
\[\|\varphi\|_{*,q} = \|\varphi\|_{L^2} + \sum_{j=1}^q \|dd^c \varphi_j\|^{1/2^j}.\]
Assume that $\varphi_q\leq 0$. Define
$$u=\sum_{n=1}^\infty 2^{n\alpha}\Big( u_n^*+\frac{\max(\varphi_q,-\lambda^n)}{\lambda^n}\Big),$$
where $u_n^*$ is the usc regularization of $u_n$. This is our desired function.
\end{proof}

\begin{example}\label{examMoser}(Sharpness of the exponential $\alpha$)
As in \cite[page 13]{Vigny-Vu-Lebesgue}, we consider the case where $\alpha > 2^q$ and $k=1$. We choose the function $\varphi$ to be $(-\log |z|^2)^{1/2^q - \delta}$ where $\delta \in (0,1/2^q)$. Then by Example~\ref{mainexample}, $\varphi \in W^*_q$. Since $\alpha > 2^q$, we can choose $\delta$ small enough such that $\beta= \alpha(1/2^q-\delta)>1$. Thus, $\varphi^\alpha = (-\log |z|^2)^\beta$ and then $e^{c\varphi^\alpha} \geq c/|z|^2$ which is not locally integrable at $0$ in $\mathbb{C}$. Also, note that by arguments in \cite{Vigny-Vu-Lebesgue}, this function is not bounded from above by minus of a subharmonic function. So, the exponential coefficient here as well as the one in Theorem~\ref{theoremA} cannot be greater than $2^q$. It is natural to predict that the result still holds for $\alpha = 2^q$ (like in the case $q=1$), but currently, we do not know how to prove that. 
\end{example}
%\subsection{Psh bound (global version)}

One can argue similarly to the proof of Theorem~\ref{main result} to obtain the following global version of psh bound. We refer the reader to \cite[Theorem 2.7]{Vigny-Vu-Lebesgue} for more details.
\begin{theorem}\label{main result global}
     Let $\varphi \in W^*_q(X)$ with $\|\varphi\|_{*,q}\leq 1$ and $\alpha \in[1,2^{q})$. Then there exist a strictly positive constant $C$ not depending on $\varphi$ and a negative $C\omega$-psh $u$ on $X$ such that
$$|\varphi|^\alpha \leq -u \qquad \text{and} \qquad\|u\|_{L^1(X)} \leq C.$$
\end{theorem}

This theorem implies the following global Moser-Trudinger inequality.

\begin{theorem}\label{theoremA'}
    Let $(X,\omega)$ be a compact K\"{a}hler manifold of dimension $k$. Let $q\geq 1$ and $\alpha \in[1,2^q)$. Let $v_1,\ldots,v_k$ be $\omega$-psh functions which are H\"{o}lder continuous of H\"{o}lder exponent $\beta$ with $\|v_j\|_{\mathscr{C}^\beta} \leq 1$ for $1\leq j \leq k$. Let $\varphi \in W^*_q(X)$ such that $\|\varphi\|_{*,q}\leq 1$. Then there exist strictly positive constants $c_1$ and $c_2$ depending on $X, \omega, \alpha$ and $\beta$ but independent of $\varphi,v_1,\ldots,v_k$ such that 
    $$\int_X e^{c_1|\varphi|^\alpha}(\omega+dd^c v_1)\wedge \cdots \wedge (\omega+dd^c v_k)\leq c_2.$$
\end{theorem}

\section{Complex Monge-Amp\`{e}re equation}\label{sec:4}
In this section, we study the relationship between $W^*_q$ and the complex Monge-Amp\`{e}re operator.

\subsection{Global setting} 
Let $(X,\omega)$ be a compact K\"{a}hler manifold of dimension $k$. We recall some definitions from \cite{GZ-weighted}. Let $\varphi$ be some unbounded $\omega$-psh function on $X$ and consider $\varphi_j= \max (\varphi,-j)$ be the canonical approximation of $\varphi$ by bounded $\omega$-psh functions. By \cite{BT_fine_87}, we can define the Monge-Amp\`{e}re measure $(\omega + dd^c \varphi_j)^k$. The sequence of measures
\[\mathbf{1}_{\left \{ \varphi > -j \right \}}\left ( \omega+dd^c \varphi_j \right )^k\]
is an increasing sequence and converges to the non-pluripolar Monge-Amp\`{e}re measure $\mu_\varphi$ of $\varphi$. Its total mass $\mu_\varphi(X)$ can take any value in $\Big[0,\int_X \omega ^k \Big]$. Define
\[\mathcal{E}(X,\omega) = \Big \{ \varphi \in \PSH(X,\omega) : \mu_\varphi(X) = \int_X \omega^k\Big \}.\]

Recall the following criterion for functions in $\mathcal{E}(X,\omega)$.
\begin{lemma}\label{criE}\cite[Lemma 1.2]{GZ-weighted} Let $\varphi \in \PSH(X,\omega)$ and define $\varphi_j= \max (\varphi,-j)$ for $j\in\N$. Let $(s_j)$ be any sequence of real numbers converging to $\infty$, such that $s_j\leq j$ for all $j\in \mathbb{N}$. Then  the following conditions are equivalent:
\begin{itemize}
    \item[(1)] $\varphi \in \mathcal{E}(X,\omega)$;
    \item[(2)] $(\omega+dd^c \varphi_j)^k (\varphi \leq -j) \rightarrow 0$;
    \item[(3)] $(\omega+dd^c \varphi_j)^k (\varphi \leq -s_j) \rightarrow 0$.
\end{itemize}
\end{lemma}
Now, we prove that any $\omega$-psh complex Sobolev function belongs to $\mathcal{E}(X,\omega)$.

\begin{proof}[Proof of Theorem~\ref{MAglobal}(1)]
    Assume that $\varphi$ is a negative $\omega$-psh function in $W^*_1(X)$ and $\varphi_j = \max(\varphi,-j)$. We note that $\varphi_j/j $ equals to $ -1$ when $\varphi \leq -j$ and equals to $\varphi/j < 0$ when $\varphi > -j$. Let $T_\varphi$ be a positive closed current on $X$ such that $d\varphi\wedge d^c\varphi\leq T_\varphi.$ By  Example~\ref{example max min}, we have $d\varphi_j\wedge d^c\varphi_j\leq T_\varphi.$
    By Stokes' formula, we have 
\begin{align*} 
\int_{\{\varphi \leq -j\}} (\omega+dd^c\varphi_j)^k &\leq -\int_{X}\frac{\varphi_j}{j} (\omega+dd^c\varphi_j)^k \\
&=\frac{1}{j} \Big(\int_X d\varphi_j \wedge d^c \varphi_j \wedge \sum_{m=0}^{k-1} (\omega + dd^c \varphi_j)^{k-1-m}\wedge \omega ^m - \int_X \varphi_j \omega^k\Big) \\
&\leq \frac{1}{j} \Big(\int_X T_\varphi \wedge \sum_{m=0}^{k-1} (\omega + dd^c \varphi_j)^{k-1-m}\wedge \omega ^m - \int_X \varphi \omega^k\Big)\cdot
\end{align*}
Since $T_\varphi$ is closed, the first integral doesn't change if we replace the closed current $\sum\limits_{m=0}^{k-1} (\omega + dd^c \varphi_j)^{k-1-m} \wedge \omega^m$ by a closed form in its de Rham cohomology class. We can replace it by $k\omega^{k-1}$ and obtain
\begin{align*}
\int_{\{\varphi \leq -j\}} (\omega+dd^c\varphi_j)^k\leq \frac{k\cdot  \|T_\varphi\|_X}{j} - \frac{1}{j} \Big(\int_X \varphi \omega^k \Big) \rightarrow 0 \text{ as } j\rightarrow \infty\cdot
\end{align*}
This, combined with Lemma~\ref{criE}, finishes our proof.
\end{proof}

Next, we recall the definition of finite energy classes in \cite{GZ-weighted}. For simplicity, assume that $\int_X \omega^ k =1$ and denote $\omega_\varphi := \omega + dd^c \varphi$. Let $p>0$ and $\varphi$ be a bounded $\omega$-psh functions, one can define the energy functional
\[E_p(\varphi) = -\frac{1}{k+1} \sum_{m=0}^{k} \int_X (-\varphi)^p (\omega_\varphi)^m \wedge \omega^{k-m}\cdot\]
We can extend this functional for arbitrary $\omega$-psh functions by canonical approximation
\[E_p(\varphi) = \lim_{j\rightarrow \infty} E_p(\varphi_j).\]
The finite energy class is defined as follows
\[\mathcal{E}^p(X,\omega) = \left \{ \varphi\in \mathcal{E}(X,\omega): E_p(\varphi)>-\infty \right \}.\]

To prove Theorem~\ref{MAglobal}(2), we first recall some facts about finite energy classes. We only use the results for the class $\mathcal{E}^p(X,\omega)$ with $p\geq1$. For more general classes and the proof of these propositions, we refer the readers to \cite{GZ-weighted} and \cite{Vu_Do-MA}.
\begin{proposition}
    Let $p\geq 1$ and  $\varphi,\psi $ be bounded non-positive $\omega$-psh functions. Then for every positive closed current $T$ of bi-dimension $(1,1)$, we have
    \[0\leq \int_X(-\varphi)^p \omega_\psi\wedge T \leq 2p\int_X (-\varphi)^p \omega_\varphi\wedge T + 2p \int_X (-\psi)^p\omega_\psi\wedge T.\]
\end{proposition}

%\noindent Recall that $\omega_\varphi = \omega +dd^c \varphi$ and $\omega_\psi = \omega +dd^c \psi$. For a proof of this proposition, see \cite[Proposition 3.6]{GZ-weighted}. 

\begin{proposition}\label{mixEp}
    Let $p\geq 1$ and $\varphi_0,\dots,\varphi_k $ be bounded non-positive $\omega$-psh functions. Then there exists a strictly positive constant $C_p$ depending only on $p$ such that
    \[\int_X (-\varphi_0)^p \omega_{\varphi_1} \wedge \cdots \wedge \omega_{\varphi_k} \leq C_p \max_{0\leq m\leq k} \Big ( \int_{X}(-\varphi_m)^p\omega_{\varphi_m}^k \Big ).\]
\end{proposition}
%\noindent See \cite[Proposition 3.8]{GZ-weighted} for a proof.

We will use the following direct corollary.

\begin{corollary}\label{compareEp}
    Let $\varphi$ and $\psi$ be non-positive functions in $\mathcal{E}^p(X,\omega)$ with $p\geq1$. Then for $0\leq m \leq k-1$, there exists a strictly positive constant $C_p$ depending only on $p$ such that
    \[\int_X (-\varphi)^p\omega_\psi \wedge \omega_\varphi^{k-m-1}\wedge \omega^m \leq C_p\Big(\int_{X}(-\varphi)^p\omega_{\varphi}^k+\int_{X}(-\psi)^p\omega_{\psi}^k\Big).\]
\end{corollary}

\noindent Here the measure $\omega_\psi \wedge \omega_\varphi^{k-m-1} \wedge \omega^m$ is defined in the non-pluripolar sense (see \cite[Chapter 10.2.3]{GZbook}.
\begin{proof}
Let $\varphi_j = \max(\varphi,-j)$ and $\psi_j = \max (\psi,-j)$. By \cite[Theorem 10.18]{GZbook},
\[\int_X (-\varphi_j)^p\omega_\psi \wedge \omega_\varphi^{k-m-1}\wedge \omega^m = \lim\limits_{j'\rightarrow \infty} \int_X (-\varphi_j)^p\omega_{\psi_{j'}} \wedge \omega_{\varphi_{j'}}^{k-m-1}\wedge \omega^m.\]
By Proposition~\ref{mixEp}, we have
\begin{align*}\int_X (-\varphi_j)^p\omega_{\psi_{j'}} \wedge \omega_{\varphi_{j'}}^{k-m-1}\wedge \omega^m &\leq C_p\Big(\int_{X}(-\varphi_j)^p\omega_{\varphi_j}^k+\int_{X}(-\varphi_{j'})^p\omega_{\varphi_{j'}}^k+\int_{X}(-\psi_{j'})^p\omega_{\psi_{j'}}^k \Big) \\
&\leq C_p\Big(\int_{X}(-\varphi)^p\omega_{\varphi}^k+\int_{X}(-\psi)^p\omega_{\psi}^k\Big).\end{align*}
Letting $j' \rightarrow \infty$ and then let $j\rightarrow \infty$ give us the result.
\end{proof}
\begin{proof}[Proof of Theorem~\ref{MAglobal}(2)]
    We prove this by induction. First, consider the case where $q=2$. Let $\psi$ be the function in $W^*_1(X)$ such that $d\varphi \wedge d^c \varphi \leq C(\omega + dd^c\psi)$. By Theorem~\ref{MAglobal}(1), we have $\varphi,\psi \in \mathcal{E}(X,\omega)$. By Stokes' formula, we have
    \begin{align*}\int_X d\varphi \wedge d^c \varphi \wedge \omega_\varphi^l \wedge \omega^{k-l-1} &= \int_X \varphi dd^c \varphi \wedge \omega_\varphi^l \wedge \omega^{k-l-1} \\ &= \int_X \varphi \omega_\varphi^{l+1}\wedge \omega^{k-l-1} - \int_X \varphi \omega_{\varphi}^{l}\wedge \varphi^{k-l}.
\end{align*}
    Thus, we can control the energy of $\varphi$ through $\psi$:
    \begin{align*} -E_1(\varphi) &= \frac{1}{k+1} \sum_{m=0}^k \int_X (-\varphi) \omega_\varphi^m \wedge \omega^{k-m}  \\ &= \int_X (-\varphi) \omega^{k} +\frac{1}{k+1} \sum_{m=1}^k \sum_{l=0}^{m-1} \int_X d\varphi \wedge d^c \varphi \wedge \omega_\varphi^l \wedge \omega^{k-l-1} \\&\leq \int_X (-\varphi) \omega^{k} + \frac{C}{k+1} \sum_{m=1}^k \sum_{l=0}^{m-1} \int_X \omega_{\psi} \wedge \omega_\varphi^l \wedge \omega^{k-l-1} \\
&= \int_X (-\varphi) \omega^{k} + \frac{C}{k+1} \sum_{m=1}^k \sum_{l=0}^{m-1} 1 < \infty\cdot
\end{align*}
So, $\varphi \in \mathcal{E}^1(X,\omega)$.

Suppose that the theorem is true for $q-1$ with $q>2$. Let $\psi$ be an $\omega$-psh function in $W^*_{q-1}(X)$ such that $d\varphi\wedge d^c \varphi \leq C(\omega +dd^c \psi)$. By the induction hypothesis, we have $\varphi,\psi \in \mathcal{E}^{q-2}(X,\omega)$.
We observe that
\[dd^c((-\varphi)^{q}/q) = (q-1)(-\varphi)^{q-2}d\varphi\wedge d^c \varphi - (-\varphi)^{q-1}dd^c \varphi.\]
By using Stokes's formula, we can write
\begin{align*} \int_X (-\varphi)^{q-1} dd^c\varphi\wedge T &= (q-1)\int_X (-\varphi)^{q-2} d\varphi \wedge d^c \varphi \wedge T - \int_X dd^c((-\varphi)^q/q)\wedge T \\
&=(q-1)\int_X(-\varphi)^{q-2}d\varphi\wedge d^c\varphi\wedge T,
\end{align*}
for $T$ is a sufficiently regular positive closed current of bi-dimension $(1,1)$. This implies that
\begin{align*}-E_{q-1}(\varphi)&=\frac{1}{k+1}\sum_{m=0}^k \int_X(-\varphi)^{q-1}(\omega_\varphi)^m\wedge \omega^{k-m} \\
&= \int_X (-\varphi)^{q-1} \omega^k+ \frac{q-1}{k+1}\sum_{m=1}^k\sum_{l=0}^{m-1} \int_X(-\varphi)^{q-2}d\varphi\wedge d^c\varphi\wedge(\omega_\varphi)^l\wedge \omega^{k-l-1} \\
&\leq \int_X (-\varphi)^{q-1} \omega^k+ \frac{C(q-1)}{k+1}\sum_{m=1}^k\sum_{l=0}^{m-1} \int_X(-\varphi)^{q-2}\omega_\psi\wedge(\omega_\varphi)^l\wedge \omega^{k-l-1}\\
&< \infty\end{align*}
by applying Corollary~\ref{compareEp} for $p=q-2$ and the induction hypothesis. So $\varphi \in \mathcal{E}^{q-1}(X,\omega)$ and we complete our proof.
\end{proof}

\begin{remark}
    The above theorem provides a lower bound for $p(q)$ by $q-1$. From Example~\ref{examMoser} and \cite[Theorem 2.1]{DGLfinite-entropy}, we see that $p(q)$ has an upper bound by $k(2^q-1)$. It is an interesting question to know which is the best option for $p(q)$. One can also ask if the envelope method in \cite{DGLfinite-entropy} can be used to give a new proof for the Moser-Trudinger inequalities. It has been noted in the beginning of Section~\ref{sec:3}, as a consequence of bounded mean oscillation property, there are constants $c$ and $A$ such that $\int_X e^{c|\varphi|} \omega^k \leq A$ for all $\varphi$ in $W^*_q(X)$ with $\|\varphi\|_{*,q}\leq 1$. This fact can be used in place of Skoda's integrability theorem in the proof of Theorem 2.1 in \cite{DGLfinite-entropy}.
\end{remark}

\subsection{Local setting}

The following result is the key point in our proof of Theorem~\ref{localMA}.
\begin{proposition}\label{key prop for D}
    Let $1\leq p\leq k-1$ and $0\leq m\leq k-p$. Let $q_1,\ldots,q_{p+1}$ be integers satisfying $q_1 \leq \cdots \leq q_{p+1}$, $q_j \geq p-1 +m$ for $1\leq j \leq p$, and $q_p \leq q_1 +1$.
    %\begin{itemize}
       %  \item[(1)] $\displaystyle q_{p+1} \geq \max_{1\leq j \leq p} q_j$,
       %  \item[(2)] $\displaystyle\sum_{j=1}^{p} q_j \geq p(p-1+m)$,
        % \item[(3)] $\displaystyle \max_{1\leq j \leq p} q_j - \min_{1\leq j\leq p} q_j \leq 1$.
   % \end{itemize}
     Let $\varphi_1,\ldots,\varphi_{p+1}$ be negative smooth psh functions on $\Omega$ such that, for $j=1,\ldots,p+1$, $\|\varphi_j\|_{*,q_j}\leq 1$. Assume that there exists a defining sequence $(\varphi_{j,1},\ldots,\varphi_{j,q_j})$ of $\varphi_j$ such that $\varphi_{j,l}$ is smooth for $l=1,\ldots,q_j$, $\|\varphi_{j,l}\|_{*,q_j-l} \leq 1$ for $l=1,\ldots,q_j-1$ and $\|\varphi_{j,q_j}\|_{L^1} \leq 1$. Then for every compact subset $K$ of $\Omega$, there is a constant $C>0$ depending only on $K,m$ and $p$ such that
      $$\int_K(-\varphi_{p+1})^m dd^c \varphi_1 \wedge \cdots \wedge dd^c \varphi_p \wedge \omega^{k-p} \leq C \text{ where } \omega = dd^c |z|^2.$$ 
\end{proposition}
    \begin{proof} We prove the result by induction on $p$ and $m$. Fix a compact set $K$ and a cut-off function $\chi$ on $\Omega$ such that $0\leq\chi\leq1$ and $\chi=1$ on $K$. In our estimates, $\lesssim$ denotes $\leq$ modulo a multiplicative constant depending only on $m$ and $K$. 
    
     Consider the case where $p=1$. If $m=0$, then the desired property follows as $\|\varphi_1\|_{*,1} \leq 1$. Assume that the desired property is true for $0,\ldots,m-1$ for some $1\leq m\leq k-p$. We now prove this for $m$. By hypothesis, $q_2\geq q_1\geq m$. Hence, there exist smooth psh functions $\psi_1, \psi_2$ satisfying $\|\psi_1\|_{*,m-1} \leq 1,\ \|\psi_2\|_{*,m-1} \leq 1 $ such that 
     \begin{equation}\label{inequa psi1 psi2}
         d\varphi_1\wedge d^c\varphi_1\leq\ddc\psi_1,\quad d\varphi_2\wedge d^c\varphi_2\leq\ddc\psi_2.
     \end{equation}
     By Stokes' formula, we have
     \begin{align}\label{induc1}
     \int_K (-\varphi_2)^m dd^c \varphi_1 \wedge  \omega^{k-1} &\lesssim  \left |\int_\Omega \chi (-\varphi_2)^{m-1}d\varphi_2 \wedge d^c \varphi_1 \wedge \omega^{k-1}\right |\\
+ &
\left |\int_\Omega (-\varphi_2)^m d\chi \wedge d^c \varphi_1 \wedge \omega^{k-1}\right |\nonumber.
\end{align}
For the first term, by using Cauchy-Schwarz inequality and inequalities (\ref{inequa psi1 psi2}), we can
bound it from above by the square root of
$$\left (\int_\Omega\chi (-\varphi_2)^{m-1} \ddc\psi_1 \wedge \omega^{k-1}\right ) \left (\int_\Omega\chi (-\varphi_2)^{m-1} \ddc\psi_2 \wedge \omega^{k-1}\right ).$$ 
It follows from the induction hypothesis for $m-1$ that both factors are bounded by a constant depending only on $\supp(\chi)$ (and hence only on $K$). Hence, the first term of the RHS of inequality \eqref{induc1} is bounded by a constant depending only on $K$.
For the second term of the RHS of inequality \eqref{induc1}, by using Cauchy-Schwarz inequality and inequalities (\ref{inequa psi1 psi2}), we can
bound it from above by the square root of
\begin{align*}
\left (\int_\Omega (-\varphi_2)^{m+1} d\chi \wedge d^c \chi \wedge \omega^{k-1}   \right )
\left (\int_{\supp (\chi)} (-\varphi_2)^{m-1}\ddc\psi_1 \wedge \omega^{k-1}  \right ).
\end{align*}
We use the induction hypothesis for $m-1$ for the second factor. The first factor can be bounded using Corollary~\ref{Cor of M-T DMV}. Hence, the second term of the RHS of inequality \eqref{induc1} is bounded by a constant depending only on $K$. The proof for the case $p=1$ is thus complete.
     
     Now, we consider the case where $p>1$. Assume that the desired property is true for $p-1$, where $2\leq p \leq k-1$. We prove that this is true for $p$. Consider the case where $m=0$. We note that 
     $$ dd^c \varphi_1 \wedge \cdots \wedge dd^c \varphi_p = dd^c \left(\varphi_p dd^c \varphi_1\wedge \cdots \wedge dd^c\varphi_{p-1}\right).$$
     Hence, by Stokes' formula, we have
     \begin{align*}\int_K dd^c \varphi_1 \wedge \cdots \wedge dd^c \varphi_p \wedge \omega^{k-p} &\leq \int_\Omega \chi dd^c \left(\varphi_p dd^c \varphi_1\wedge \cdots \wedge dd^c\varphi_{p-1}\right) \wedge \omega^{k-p} \\
     &= \int_\Omega (-\varphi_p) dd^c\chi \wedge dd^c \varphi_1\wedge \cdots \wedge dd^c\varphi_{p-1} \wedge \omega^{k-p} \\
     &\lesssim \int_{\supp(\chi)} (-\varphi_p)dd^c \varphi_1\wedge \cdots \wedge dd^c\varphi_{p-1} \wedge \omega^{k-p+1}. \end{align*}
     We now only need to check that $\varphi_1,\ldots,\varphi_p$ satisfy the induction hypothesis for $p-1$ and $m=1$. This is clear because $q_{p-1} \leq q_p \leq q_1+1$ and $q_j \geq p-1 = (p-1)-1+1$ for $1\leq j \leq p-1$. Therefore, we get the statement for $m=0$.
    
    Assume that the statement is true for $0,\ldots,m-1$, where $1\leq m\leq k-p$. We now prove this for $m$.
    Since $q_{p+1}\geq q_p$, there exist smooth psh functions $\phi_p,\phi_{p+1}$ satisfying $\|\phi_p\|_{*,q_p-1} \leq 1, \|\phi_{p+1}\|_{*,q_p-1} \leq 1 $ such that
     \begin{equation}\label{inequa phi1 phip+1}
         d\varphi_p\wedge d^c\varphi_p\leq\ddc\phi_p,\quad d\varphi_{p+1}\wedge d^c\varphi_{p+1}\leq\ddc\phi_{p+1}.
     \end{equation}
     We need to bound, for $\varphi = \varphi_{p+1}$,
     $$\int_\Omega \chi(-\varphi)^m dd^c \varphi_1 \wedge \cdots \wedge dd^c \varphi_p \wedge \omega^{k-p}.$$
     By Stokes' formula, we have
    \begin{align}\label{induc2}
    &\int_\Omega \chi(-\varphi)^m dd^c \varphi_1 \wedge \cdots \wedge dd^c \varphi_p \wedge \omega^{k-p}\\ &\lesssim \left |\int_\Omega (-\varphi)^m d\chi \wedge d^c \varphi_p \wedge dd^c \varphi_1 \wedge \cdots\wedge dd^c \varphi_{p-1} \wedge \omega^{k-p}\right | \nonumber\\
&+\left |\int_\Omega \chi (-\varphi)^{m-1}d\varphi \wedge d^c \varphi_p \wedge dd^c \varphi_1 \wedge \cdots\wedge dd^c \varphi_{p-1} \wedge \omega^{k-p}
\right |\nonumber.
\end{align}
For the second term of the RHS of inequality (\ref{induc2}), by using Cauchy-Schwarz inequality and inequalities \eqref{inequa phi1 phip+1}, we can bound it from above by the square root of
\begin{align*}
    &\left (\int_\Omega \chi (-\varphi)^{m-1}\ddc \phi_p \wedge dd^c \varphi_1 \wedge \cdots\wedge dd^c \varphi_{p-1} \wedge \omega^{k-p}\right )\\
&\times \left (\int_\Omega \chi (-\varphi)^{m-1}\ddc\phi_{p+1} \wedge dd^c \varphi_1 \wedge \cdots\wedge dd^c \varphi_{p-1} \wedge \omega^{k-p}\right ).
\end{align*}
We note that $q_p-1, q_1,\ldots,q_{p-1},q_{p+1}$ satisfy induction hypothesis for $p$ and $m-1$. Thus, both factors are bounded by a constant depending only on $\supp(\chi)$ (and hence only on $K$). Hence, the second term of the RHS of inequality \eqref{induc2} is bounded by a constant depending only on $K$.
For the first term of the RHS of inequality \eqref{induc2}, by using Cauchy-Schwarz inequality and inequalities (\ref{inequa phi1 phip+1}), we can bound it from above by the square root of
\begin{align}\label{induc3}&\left (\int_\Omega (-\varphi)^{m+1} d\chi \wedge d^c \chi \wedge dd^c \varphi_1 \wedge \cdots\wedge dd^c \varphi_{p-1} \wedge \omega^{k-p}   \right ) \\ &\times
\left (\int_{\supp (\chi)} (-\varphi)^{m-1}\ddc\phi_p \wedge dd^c \varphi_1 \wedge \cdots\wedge dd^c \varphi_{p-1} \wedge \omega^{k-p}  \right ). \nonumber
\end{align}
It is clear that $ q_1,\ldots,q_{p-1},q_{p+1}$ satisfy induction hypothesis for $p-1$ and $m+1$ and $q_p-1, q_1,\ldots,q_{p-1},q_{p+1}$ satisfy induction hypothesis for $p$ and $m-1$. Thus, we can bound \eqref{induc3} by a constant depending only on $K$. Therefore, we can bound the first term of the RHS of inequality \eqref{induc2}. The proof is complete.
\end{proof}
  We now prove that psh $q$-complex Sobolev functions belong to $\mathcal{D}(\Omega)$ for $q\geq k-1$.

\begin{proof}[End of proof of Theorem~\ref{localMA}]
Let $\varphi \in W^*_{k-1,\text{loc}}(\Omega)\cap\psh(\Omega)$. Since the problem is local, we can assume that $\varphi \in W^*_{k-1}(\Omega)\cap\psh(\Omega)$. We will prove that $\varphi$ satisfies the condition \eqref{condlocalMA}. Let $U$ be an open relatively compact subset of $\Omega$ and $(\varphi_n)$ be the sequence of smooth psh functions constructed in the proof of Theorem~\ref{localdense}. We note that $\|\varphi_{n}\|_{W^*_{k-1}(U)}\leq 2\|\varphi\|_{W^*_{k-1}(\Omega)}$ for every $n.$
Let $(\varphi_{0,1},\ldots,\varphi_{0,k-1})$ be a smooth defining sequence for $\varphi$. By the proof of Theorem~\ref{localdense}, we can construct a smooth defining sequence $(\varphi_{n,1},\ldots,\varphi_{n,k-1})$ for $\varphi_n$ such that
 $\|\varphi_{n,l}\|_{W^*_{k-1-l}(U)}\leq 2\|\varphi_{0,l}\|_{W^*_{k-1-l}(\Omega)}$ for $1\leq l \leq k-2$, $\|\varphi_{n,k-1}\|_{L^1(U)}\leq 2\|\varphi_{0,k-1}\|_{L^1(\Omega)}$, for every $n$. Rescale $\varphi$ if necessary, we can assume that, for every $n$, $\|\varphi_n\|_{W^*_{k-1}(U)} \leq 1$, $ \|\varphi_{n,l}\|_{W^*_{k-1-l}(U)} \leq 1$ for $1\leq l\leq k-2$ and $\|\varphi_{n,k-1}\|_{L^1(U)} \leq 1$.

 Since $d\varphi_n\wedge d^c\varphi_n\leq\ddc \varphi_{n,1},$ we have
\begin{align*}
    &\int_U |\varphi_n|^{k-p-2} d\varphi_n\wedge d^c\varphi_n \wedge (dd^c \varphi_n)^{p} \wedge \omega^{k-p-1} \\
    &\leq \int_U|\varphi_n|^{k-p-2} dd^c \varphi_{n,1} \wedge (dd^c \varphi_n)^{p} \wedge \omega^{k-p-1}.
\end{align*}
 It follows from Proposition~\ref{key prop for D} that the right-hand side is uniformly bounded by a constant depending only on $U$. Therefore, $ |\varphi_n|^{k-p-2} d\varphi_n\wedge d^c\varphi_n \wedge (dd^c \varphi_n)^{p} \wedge \omega^{k-p-1}$ are locally weakly bounded in $U$ for every $p = 0,\ldots,k-2$. The proof is complete.
\end{proof}

\bibliography{biblio_family_MA,biblio_Viet_papers,bib-kahlerRicci-flow,bib_newpapers}
\bibliographystyle{alpha}

\bigskip

\end{document}